\documentclass[12pt]{amsart}
\usepackage{amsfonts, amssymb, latexsym, graphicx, hyperref, amsmath}
\usepackage[vcentermath]{youngtab}
\usepackage{tikz}
\usepackage{tikz-cd}
\usetikzlibrary{calc, shapes, backgrounds,arrows,positioning,decorations.pathmorphing}

\newtheorem{theorem}{Theorem}[section]
\newtheorem{proposition}[theorem]{Proposition}

\newtheorem{lemma}[theorem]{Lemma}

\newtheorem{claim}[theorem]{Claim}

\newtheorem*{claim*}{Claim}
\newtheorem{corollary}[theorem]{Corollary}
\newtheorem{Main Conjecture}[theorem]{Main Conjecture}
\newtheorem{conjecture}[theorem]{Conjecture}
\newtheorem{problem}[theorem]{Problem}
\theoremstyle{remark}

\newtheorem{example}[theorem]{Example}

\theoremstyle{plain}

\newcommand{\cellsize}{11}
\newlength{\cellsz} 
\newsavebox{\cell}
\sbox{\cell}{\begin{picture}(\cellsize,\cellsize)
\put(0,0){\line(1,0){\cellsize}}
\put(0,0){\line(0,1){\cellsize}}
\put(\cellsize,0){\line(0,1){\cellsize}}
\put(0,\cellsize){\line(1,0){\cellsize}}
\end{picture}}
\newcommand\cellify[1]{\def\thearg{#1}\def\nothing{}%
\ifx\thearg\nothing
\vrule width0pt height\cellsz depth0pt\else
\hbox to 0pt{\usebox{\cell} \hss}\fi%
\vbox to \cellsz{
\vss
\hbox to \cellsz{\hss$#1$\hss}
\vss}}
\newcommand\tableau[1]{\vtop{\let\\\cr
\baselineskip -16000pt \lineskiplimit 16000pt \lineskip 0pt
\ialign{&\cellify{##}\cr#1\crcr}}}

\newcommand{\gap}{\hspace{1in} \\ \vspace{-.2in}}

\newcommand{\excise}[1]{}

\title{Newell-Littlewood numbers}
\author{Shiliang Gao}
\author{Gidon Orelowitz}
\author{Alexander Yong}
\address{Dept.~of Mathematics, University of Illinois at Urbana-Champaign, Urbana, IL 61801}

\email{sgao23@illinois.edu, gidono2@illinois.edu, ayong@illinois.edu}

\date{May 18, 2020}
\begin{document}
\pagestyle{plain}
\maketitle

\vspace{-.2in}
\begin{abstract}
The \emph{Newell-Littlewood numbers} are defined in terms of their celebrated cousins,
the \emph{Littlewood-Richardson coefficients}. Both arise as tensor
product multiplicities for a classical Lie group. They are the structure coefficients of the K.~Koike-I.~Terada basis
of the ring of symmetric functions. Recent work of H.~Hahn studies them, motivated by R.~Langlands' \emph{beyond endoscopy}
proposal; we address her work with a simple
characterization of \emph{detection} of Weyl modules. This motivates further study of the combinatorics of the numbers.
We consider analogues of ideas of 
J.~De Loera-T.~McAllister,
H.~Derksen-J.~Weyman, S.~Fomin--W.~Fulton-C.-K.~Li--Y.-T.~Poon, W.~Fulton, R.~King-C.~Tollu-F.~Toumazet, 
M.~Kleber,
A.~Klyachko, A.~Knutson-T.~Tao, T.~Lam-A.~Postnikov-P.~Pylyavskyy,
K.~Mulmuley-H.~Narayanan-M.~Sohoni,
H.~Narayanan, A.~Okounkov, J.~Stembridge, and H.~Weyl.
 
\end{abstract}

\tableofcontents

\section{Introduction}

\subsection{Overview}
The \emph{Newell-Littlewood numbers} \cite{Newell, Littlewood} are defined as 
\begin{equation}
\label{eqn:Newell-Littlewood}
N_{\mu,\nu,\lambda}=\sum_{\alpha,\beta,\gamma} c_{\alpha,\beta}^{\mu}
c_{\alpha,\gamma}^{\nu}c_{\beta,\gamma}^{\lambda},
\end{equation}
where the indices are partitions 
in 
\[{\sf Par}_n=\{(\lambda_1, \lambda_2, \ldots, \lambda_n)\in {\mathbb Z}_{\geq 0}^n:
\lambda_1\geq \lambda_2\geq \ldots \geq \lambda_n\}.\]
Here, $c^{\mu}_{\alpha,\beta}$ is the  
\emph{Littlewood-Richardson coefficient}. The latter numbers are of interest in combinatorics, representation theory and algebraic geometry; see, \emph{e.g.}, the 
books \cite{Fulton.Harris, FultonYT, ECII}. We study $N_{\mu,\nu,\lambda}$ by  analogy with modern research on their better known constituents.

For an $n$-dimensional complex vector space $V$ over ${\mathbb C}$ and $\lambda\in {\sf Par}_n$, the \emph{Weyl module} (or \emph{Schur functor}) ${\mathbb S}_{\lambda}(V)$ is an irreducible ${\sf GL}(V)$-module (\cite[Lectures~6 and 15]{Fulton.Harris} is our reference).
The Littlewood-Richardson coefficient is the tensor product multiplicity
\begin{equation}
\label{eqn:GLexp}
{\mathbb S}_{\mu}(V) \otimes {\mathbb S}_{\nu}(V)\cong\bigoplus_{\lambda \in {\sf Par}_n}
{\mathbb S}_{\lambda}(V)^{\oplus c_{\mu,\nu}^{\lambda}}.
\end{equation}

The Newell-Littlewood numbers arise in a similar manner, where ${\sf GL}(V)$ is replaced
by one of the other classical Lie groups ${\sf G}$. That is, suppose $W$ 
is a complex vector space, with a nondegenerate symplectic or orthogonal form $\omega$,
where $\dim W = 2n+ \delta$ and $\delta\in \{0,1\}$. Fix a basis $\{\varepsilon_1,
\varepsilon_2,\ldots,\varepsilon_{2n+\delta}\}$ such that 
\[\omega(\varepsilon_k,
\varepsilon_{2n+1+\delta-k})=\pm \omega(\varepsilon_{2n+1+\delta-k},\varepsilon_k)=1, 
\text{\ \ if $1\leq i\leq n+\delta$}\] 
(other pairings are zero). Let ${\sf G}$
be the subgroup of ${\sf SL}(W)$ preserving $\omega$. Then ${\sf G}={\sf SO}_{2n+1}$
if $\dim W=2n+1$ and $\omega$ is orthogonal. It is ${\sf G}={\sf Sp}_{2n}$ if $\dim W=2n$
and $\omega$ is symplectic. Finally, ${\sf G}={\sf SO}_{2n}$ if $\dim W=2n$ and $\omega$
is once again orthogonal. These are, respectively, groups in the $B_n, C_n, D_n$ 
series of the Cartan-Killing classification.

If $\lambda\in {\sf Par}_n$, H.~Weyl's construction \cite{Weyl} (see also
\cite[Lectures~17 and 19]{Fulton.Harris})
gives
 a ${\sf G}$-module ${\mathbb S}_{[\lambda]}(W)$. 
In the \emph{stable range} $\ell(\mu)+\ell(\nu)\leq n$, 
\begin{equation}
\label{eqn:NLexpansion}
{\mathbb S}_{[\mu]}(W)\otimes  {\mathbb S}_{[\nu]}(W) \cong\bigoplus_{\lambda\in {\sf Par}_n}
{\mathbb S}_{[\lambda]}(W)^{\oplus N_{\mu,\nu,\lambda}};
\end{equation}
this is \cite[Corollary~2.5.3]{Koike}. ${\mathbb S}_{[\lambda]}(W)$ is an irreducible
${\sf G}$-module, except in type $D_n$, where irreducibility holds if $\lambda_n=0$
(otherwise it is the direct sum of two irreducible ${\sf G}$-modules).

For any semisimple connected complex algebraic group ${\sf G}$ there is
an irreducible ${\sf G}$-module $V_{\lambda}$ for each dominant weight $\lambda$.  Uniform-type combinatorial frameworks for tensor product multiplicities (subsuming $c_{\lambda,\mu}^{\nu}$ and $N_{\mu,\nu,\lambda}$) are central in combinatorial representation theory; see, \emph{e.g.}, the 
surveys \cite{Barcelo.Ram, Kumar2010} for details and references. To compare and contrast, 
$N_{\mu,\nu,\lambda}$ is \emph{itself} independent of the choice of ${\sf G}$ \cite[Theorem~2.3.4]{Koike}.

Our thesis is that, like the Littlewood-Richardson coefficients, the Newell-Littlewood numbers form a subfamily of  the general multiplicities whose combinatorics deserves separate study. Indeed, we reinforce the parallel with the Littlewood-Richardson coefficients by developing the topic from first principles and symmetric function basics. 

\subsection{Earlier work}\label{sec:earlierwork} Reading includes
K.~Koike-I.~Terada's \cite{Koike} which cites D.~E.~Littlewood's book \cite{Littlewoodbook} and
R.~C.~King's \cite{King70, King75}. In turn, \cite{King70, King75} reference the papers
of M.~J.~Newell \cite{Newell} and D.~E.~Littlewood \cite{Littlewood}.
The Schur function $s_{\lambda}$, an element of the ring
$\Lambda$ of symmetric functions, is the ``universal character'' of ${\mathbb S}_{\lambda}(V)$. By analogy, \cite[Section~2]{Koike} establishes universal characters 
of ${\mathbb S}_{[\lambda]}(W)$ for the other classical groups. 

In addition, \cite[Theorem~2.3.4]{Koike} shows that, in the stable range,  the tensor
product multiplicities coincide across the classical Lie groups (of types $B,C,D$). For definiteness, we discuss ${\sf Sp}$. It has a universal character basis 
$\{s_{[\lambda]}\}$ of $\Lambda$ 
such that
\begin{equation}
\label{eqn:polyver}
s_{[\mu]}s_{[\nu]}=\sum_{\lambda}N_{\mu,\nu,\lambda}s_{[\lambda]},
\end{equation}
where $\mu,\nu,\lambda$ are arbitrary partitions; we call this the \emph{Koike-Terada} basis.\footnote{\cite{Koike} defines another basis, for ${\sf SO}$. It  also has $N_{\mu,\nu,\lambda}$ as its structure coefficients \cite[Theorem~2.3.4 (3)]{Koike}. Hence, for our purposes, discussing
${\sf Sp}$ rather than the ${\sf SO}$ basis is merely a matter of choice.}
This basis specializes to the characters for fixed ${\sf Sp}_{2n}$, just as the specialization 
\begin{equation}
\label{eqn:thespecialization}
s_{\lambda}\mapsto s_{\lambda}(x_1,x_2,\ldots,x_n,0,0,\ldots)
\end{equation}
does for ${\sf GL}_n$. Their work discusses ``modification rules'' (cf.~\cite{King70, King75}) to \emph{non-positively} compute multiplicities outside the stable range. See \cite{Kwon} for recent work connecting the stable range
combinatorics to crystal models in combinatorial representation theory.

This paper does not focus on the Koike-Terada basis \emph{per se}.~It is devoted to the inner logic of the Newell-Littlewood numbers. We were inspired by  H.~Hahn's \cite{Hahn} which concerns the
case $\mu=\nu=\lambda$; we engage her work in Section~\ref{sec:Hahn}.

\subsection{Summary of results} 
Section~\ref{sec:prelim} collects elementary facts  about $N_{\mu,\nu,\lambda}$ (Lemma~\ref{prop:easy}).
We will need a Pieri-type rule (Proposition~\ref{thm:Pieri}).
 This appears as S.~Okada's \cite[Proposition~3.1]{Okada} 
 with a short derivation from (\ref{eqn:Newell-Littlewood}) (which we include for completeness); see also earlier work of A.~Berele \cite{Berele} and S.~Sundaram \cite{Sundaram}.

In Section~\ref{sec:product}, we derive our initial result:
\begin{itemize}
\item[(I)] Theorem~\ref{thm:interval} describes the ``shape'' of (\ref{eqn:polyver}). It characterizes the sizes of 
$\lambda$ that appear in (\ref{eqn:polyver}) and gives a comparison result for partitions of different sizes. This result suggests the Unimodality Conjecture~\ref{conj:log}.
\end{itemize}

Section~\ref{sec:Hahn} is about the original stimulus for our work. We address
a combinatorial question of H.~Hahn \cite{Hahn} (who was motivated
by R.~Langlands' \emph{beyond endoscopy} proposal \cite{Beyond} towards his 
functoriality conjecture \cite{Langlands}). More specifically, we prove
\begin{itemize}
\item[(II)] Theorem~\ref{thm:main}, which is equivalent to showing
\[N_{\lambda,\lambda,\lambda}>0\text{\  if and only if $|\lambda|\equiv 0 \ ({\mathrm{mod}}\  2)$.}\]
In \cite{Hahn}, ``$\Rightarrow$'' was proved (see~Lemma~\ref{prop:easy}(V)) and the ``$\Leftarrow$'' implication was established for three infinite families
of $\lambda$.
\end{itemize}

In Section~5, suggested by the simplicity of (II), we develop a broader framework by investigating ``polytopal'' aspects of (\ref{eqn:Newell-Littlewood}).
\begin{itemize}
\item[(III)] Theorem~\ref{thm:thepolytope} shows that $N_{\mu,\nu,\lambda}$ counts
the number of lattice points in a polytope ${\mathcal P}_{\mu,\nu,\lambda}$ that
we directly construct (avoiding use of \cite{BZ}). Its Corollary~\ref{prop:NLsemigroup} says that 
\[{\sf NL}_n:=\{(\mu,\nu,\lambda)\in {\sf Par}_n^3: N_{\mu,\nu,\lambda}>0\}\]
 is a semigroup. 
\item[(IV)] We state two logically equivalent saturation conjectures about ${\sf NL}_n$, \emph{i.e.},
Conjectures~\ref{theconj} and~\ref{theconj'}. We prove special cases (Corollary~\ref{cor:triplelambdas},
Theorem~\ref{thm:satspecial}, Corollary~\ref{n=2sat}). While saturation holds for the Littlewood-Richardson coefficients  \cite{KT}, it does \emph{not} hold for the general tensor product multiplicities (although it is conjectured for simply-laced types). The aforementioned results and conjectures provide a new view on this subject (compare, \emph{e.g.}, \cite{Kumar, Kiers} and the references therein).

\item[(V)] Among the Horn inequalities \cite{Horn} are the \emph{Weyl inequalities} 
\cite{Weyl1912}. Our ``extended Weyl inequalities'' hold whenever $N_{\mu,\nu,\lambda}>0$;  this is 
Theorem~\ref{thm:Horn}. Theorem~\ref{prop:r=2,def=k} is our justification of the nomenclature; it establishes that the (extended) Weyl inequalities are enough to characterize ${\sf NL}_2$. Our proof uses a generalizable strategy; we will return to this in a sequel.
\item[(VI)] We also discuss limits of the analogy with $c_{\mu,\nu}^{\lambda}$.  Theorem~\ref{prop:floorex} shows that R.~C.~King-C.~Tollu-F.~Toumazet's Littlewood-Richardson polynomial conjecture \cite{King} (proved by H.~Derksen-J.~Weyman \cite{Derksen.Weyman}) has no na\"ive Newell-Littlewood version.
\item[(VII)] Section~\ref{sec:complexity} sketches the computational complexity
implications of Theorem~\ref{thm:thepolytope}. 
\end{itemize}

The ``nonvanishing'' results of Section~\ref{sec:polytope} are related to Section~\ref{sec:multfree}, where we  prove:
\begin{itemize}
\item[(VIII)] Theorem~\ref{thm:mfrep}, which characterizes pairs $(\lambda,\mu)$ 
such that (\ref{eqn:polyver}) is multiplicity-free. This is an analogue of J.~R.~Stembridge's \cite[Theorem~3.1]{jrs} for Schur functions, with a similar, self-contained proof. 
\end{itemize}

Section~\ref{sec:final} gathers some miscellaneous items. This includes two open problems, and 
\begin{itemize}
\item[(IX)] Theorem~\ref{thm:meetjoin}, which generalizes results of
T.~Lam-A.~Postnikov-P.~Pylyavskyy \cite{AJM} that solved conjectures of
A.~Okounkov \cite{Okounkov} and of S.~Fomin-W.~Fulton-C.-K.~Li-Y.-T.~Poon~
\cite{FominPoon}.
\end{itemize}

The appendix gives a list of decompositions (\ref{eqn:polyver}) for the reader's convenience. 

\section{Preliminaries}\label{sec:prelim}

\subsection{The Littlewood-Richardson rule}\label{sec:LR}
Let ${\sf Par}$ be the set of all partitions (with parts of size $0$ being ignored).
Identify $\lambda\in {\sf Par}$ with Young diagrams of shape $\lambda$ (drawn in English convention). 
Let $\ell(\lambda)$ be the number of nonzero parts of $\lambda$ and let
$|\lambda|:=\sum_{i=1}^{\ell(\lambda)} \lambda_i$ 
be the \emph{size} of $\lambda$, that is, the number of boxes of $\lambda$.
If $\mu\subseteq\lambda$, the \emph{skew
shape} $\lambda/\mu$ is the set-theoretic difference of the diagrams when aligned by their northwest most box. 

 A \emph{semistandard filling} $T$ of $\lambda/\mu$ assigns positive integers to each box of $\lambda/\mu$ such that the rows are
weakly increasing from left to right, and the columns are strictly increasing from top
to bottom. The \emph{content} of $T$ is $(c_1,c_2,\ldots)$ where
$c_i=\#\{i\in T\}$. Let 
\[{\sf rowword}(T)=(w_1,w_2,\ldots,w_{|\lambda/\mu|})\]
be the right to left, top to bottom, row reading word of $T$. We say
${\sf rowword}(T)$ is \emph{ballot} if for each $i,k\geq 1$ we have
\[\#\{w_j=i:j\leq k\}\geq \#\{w_j=i+1:j\leq k\}.\]
$T$ is \emph{ballot} if ${\sf rowword}(T)$ is ballot. The \emph{Littlewood-Richardson coefficient} $c_{\mu,\nu}^{\lambda}$ is the number of ballot, 
semistandard tableaux of shape $\lambda/\mu$ and content $\nu$; we will call these \emph{LR tableaux}.

\begin{example}
If $\mu=(3,1), \nu=(4,2,1), \lambda=(5,4,2)$ then $c_{\mu,\nu}^{\lambda}=2$ because of
these two tableaux:
\[T_1=\tableau{\ & \ &\ &1&1\\ \ &1&2&2\\1 &3} \text{\ \ \ and \ \ \  }
T_2=\tableau{\ & \ &\ &1&1\\ \ &1&1&2\\2 &3}
\]
Here ${\sf rowword}(T_1)=(1,1,2,2,1,3,1)$ and ${\sf rowword}(T_2)=(1,1,2,1,1,3,2)$.\qed
\end{example}

The Littlewood-Richardson rule implies that $N_{\mu,\nu,\lambda}$ is well-defined for 
$\mu,\nu,\lambda\in {\sf Par}$.

\subsection{Facts about $N_{\mu,\nu,\lambda}$}
We gather some simple facts we will use; we make no claims of originality:
\begin{lemma}[Facts about the Newell-Littlewood numbers]\label{prop:easy}
\gap
\begin{itemize}
\item[(I)] $N_{\mu,\nu,\lambda}$ is invariant under any ${\mathfrak S}_3$-permutation of 
the indices $(\mu,\nu,\lambda)$.
\item[(II)] $N_{\mu,\nu, \lambda}=c_{\mu,\nu}^{\lambda}$ if $|\mu|+|\nu|=|\lambda|$.
\item[(III)] $N_{\mu,\nu, \lambda}=0$ unless $|\mu|,|\nu|,|\lambda|$ satisfy the triangle inequalities (possibly with equality), \emph{i.e.,} $|\mu|+|\nu|\geq |\lambda|$, 
$|\mu|+|\lambda|\geq |\nu|$, and $|\lambda|+|\nu|\geq |\mu|$.\footnote{In the case of
\emph{reduced Kronecker coefficients} ${\overline g}_{\mu,\nu}^{\lambda}$ these are
called \emph{Murnaghan's inequalities}.}
\item[(IV)] $N_{\mu,\nu, \lambda}=0$ if $|\nu\wedge \lambda|+|\mu\wedge\nu|<|\nu|$.\footnote{Recall $\nu\wedge\lambda$ is the partition whose $i$-th part  is $\min(\nu_i,\lambda_i)$.} 
\item[(V)] $N_{\mu,\nu, \lambda}=0$ unless $|\lambda|+|\mu|+|\nu|\equiv 0 \ (\!\!\!\mod 2)$.
\item[(VI)] $N_{\mu,\nu,\lambda}=N_{\mu',\nu',\lambda'}$ where $\mu'$ is the conjugate partition of $\mu$, \emph{etc.} 
\end{itemize}
\end{lemma} 
\begin{proof}
(I) is immediate from (\ref{eqn:Newell-Littlewood}).

By (\ref{eqn:Newell-Littlewood}), $N_{\mu,\nu,\lambda}=0$ unless there exist $\alpha,\beta,\gamma\in {\sf Par}$ such that $c_{\alpha,\beta}^{\mu},c_{\alpha,\gamma}^{\nu}, c_{\beta,\gamma}^{\lambda}>0$. Henceforth we will call $\alpha,\beta,\gamma$ a \emph{witness} for $N_{\mu,\nu,\lambda}>0$. These Littlewood-Richardson
coefficients are zero unless 
\[|\alpha|+|\beta|=|\mu|, |\alpha|+|\gamma|=|\nu|,
|\beta|+|\gamma|=|\lambda| \text{\ (respectively).}\] 
Therefore 
\begin{equation}
\label{eqn:parity}
2|\alpha|+|\lambda|=|\mu|+|\nu|,
\end{equation}
which implies
$|\lambda|\leq |\mu|+|\nu|$. Now apply (I) to get (III). If $|\lambda|=|\mu|+|\nu|$ then (\ref{eqn:parity}) implies the only witness is $\alpha=\emptyset,\beta=\mu,\gamma=\nu$, hence $N_{\mu,\nu,\lambda}=c_{\mu,\nu}^{\lambda}$, as asserted by (II). 

For (IV), any such $\gamma$ satisfies
$\gamma\subseteq \nu,\lambda$. Hence $|\gamma|\leq |\nu\wedge\lambda|$. Similarly,
$|\alpha|\leq |\mu\wedge\nu|$. Now combine these inequalities with the fact that
$|\alpha|+|\gamma|=|\nu|$.

(V) holds by (\ref{eqn:parity}).

Finally, (VI) holds by the standard fact $c_{\alpha,\beta}^{\mu}=c_{\alpha',\beta'}^{\mu'}, c_{\alpha,\gamma}^{\nu}=c_{\alpha',\gamma'}^{\nu'}$ and $c_{\beta,\gamma}^{\lambda}=c_{\beta',\gamma'}^{\lambda'}$.
\end{proof}

\subsection{Symmetric functions}
Let $\Lambda$ be the ring of symmetric functions in $x_1,x_2,\ldots$.
Define the (skew) \emph{Schur function}
\[s_{\mu/\lambda}(x_1,x_2,\ldots) :=\sum_{T} x^T,\]
where the sum is over semistandard Young tableaux of skew shape $\mu/\lambda$.

It is true that $s_{\mu/\lambda}\in \Lambda$. Moreover, the $\{s_{\lambda}:\lambda\in {\sf Par}, |\lambda|=N\}$
 is a basis of $\Lambda^{(N)}$, the degree $N$ homogeneous component of 
$\Lambda=\bigoplus_N \Lambda^{(N)}$. 
In fact,
\begin{equation}
\label{eqn:skew}
s_{\lambda/\mu}=\sum_{\nu} c_{\mu,\nu}^{\lambda} s_{\nu},
\end{equation}
and
\begin{equation}
\label{eqn:prodabc}
s_{\mu} s_{\nu}=\sum_{\lambda} c_{\mu,\nu}^{\lambda} s_{\lambda}.
\end{equation}
There is an inner product $\langle \cdot,\cdot\rangle:\Lambda\times \Lambda\to {\mathbb Q}$ such that $\langle s_{\lambda}, s_{\mu}\rangle =\delta_{\lambda,\mu}$; see \cite[Chapter~7]{ECII}.

We will make use of the following \emph{asymmetric} formula for $N_{\mu,\nu,\lambda}$:

\begin{proposition}
\label{prop:reformulate}
$N_{\mu,\nu,\lambda}=\sum_{\alpha} \langle s_{\mu/\alpha}s_{\nu/\alpha}, s_{\lambda}\rangle$, where the sum is over $\alpha\subseteq \mu \wedge \nu$. 
\end{proposition}
\begin{proof}
Combine (\ref{eqn:skew}), (\ref{eqn:prodabc}) and (\ref{eqn:Newell-Littlewood}) with
the fact that $s_{\mu/\alpha}=0$ unless $\alpha\subseteq \mu$ and $s_{\nu/\alpha}=0$
unless $\alpha\subseteq \nu$.
\end{proof}

Although we will not need it in this paper, we recall the definition of
$s_{[\lambda]}$ from \cite[Definition~2.1.1]{Koike}. Let $h_t=s_{(t)}$ be the homogeneous symmetric function of degree $t$. If $t<0$ then by convention $h_t=0$. Then if $\lambda\in {\sf Par}_n$,
let $\lambda^*=(\lambda_1, \lambda_2-1,\ldots,\lambda_{n}-(n-1))$. Below, $h_{\lambda^*}$ denotes
the column vector $(h_{\lambda_1},h_{\lambda_2-1},\ldots,h_{\lambda_n-(n-1)})^t$ and
$h_{\lambda^*+j(1^n)}+h_{\lambda^*-j(1^n)}$ means the column vector
\[(h_{\lambda_1+j}+h_{\lambda_1-j},h_{\lambda_2-1+j}+h_{\lambda_2-1-j},\ldots,
h_{\lambda_i-(i-1)+j}+h_{\lambda_i-(i-1)-j}, \ldots, h_{\lambda_n-(n-1)+j}
+h_{\lambda_n-(n-1)-j}
)^t.\]
With this notation,
\[s_{[\lambda]}:=
\left| h_{\lambda^*} \  h_{\lambda^*+(1^n)}+h_{\lambda^*-(1^n)} \  \cdots
\ h_{\lambda^*+j(1^n)}+h_{\lambda^*-j(1^n)} \ \cdots \ h_{\lambda^*+(n-1)(1^n)}
+h_{\lambda^*-(n-1)(1^n)}\right|.
\]
Hence, for example
\[s_{[4,2,1]}=
\left|
\begin{matrix}
h_4 & h_5+h_3 & h_6+h_2\\
h_1 & h_2 +1 & h_3\\
0 & 1 & h_1
\end{matrix}\right|=s_{4,2,1}-s_{4,1}-s_{3,2}-s_{3,1,1}+s_3+s_{2,1}.\]

\subsection{Pieri rules}\label{sec:Pieri}
The Pieri rule for Schur functions \cite[Theorem~7.5.17]{ECII}  states that
\begin{equation}
\label{eqn:usualPieri}
s_\mu s_{(p)} = \sum_{\lambda} s_\lambda,
\end{equation}
 where the sum is over all $\lambda$ such that $\lambda/\mu$ consists of $p$ boxes, none of which are in the same column.  We need the Newell-Littlewood analogue. It was known, and we include
 a proof which is the same as \cite[Proposition~3.1]{Okada} for completeness:

\begin{proposition}[Pieri-type rule; Theorem~13.1 of \cite{Sundaram} and Proposition~3.1 of \cite{Okada}]
\label{thm:Pieri}
$N_{\mu,(p),\lambda}$ equals the number of ways to remove $\frac{|\mu|+p-|\lambda|}{2}$ boxes from $\mu$ (all from different columns), then add $\frac{|\lambda|+p-|\mu|}{2}$ boxes (all to different columns) to make $\lambda$. In other words,
\begin{equation}
\label{eqn:inotherwords}
{s}_{[\mu]}{s}_{[(p)]}=\sum_{\lambda} {s}_{[\lambda]},
\end{equation}
where the sum is over the multiset of $\lambda$ obtained from $\mu$ by removing a horizontal strip of $j$ boxes where $0\leq j\leq p$ and then adding a horizontal strip of length $p-j$ boxes. 
\end{proposition}
\begin{proof}
Consider any $\alpha,\beta, \gamma$ such that $c_{\alpha,\beta}^\mu c_{\alpha,\gamma}^{(p)}c_{\beta,\gamma}^\lambda >0$.  By (\ref{eqn:parity}),
$2|\alpha| = |\mu|+p - |\lambda|$, so $|\alpha| = \frac{|\mu|+p-|\lambda|}{2}$ and similarly $|\gamma|= \frac{|\lambda|+p-|\mu|}{2}$.  Since $\alpha, \gamma \subseteq (p)$, we have that $\alpha = (\frac{|\mu|+p-|\lambda|}{2})$ and $\gamma= (\frac{|\lambda|+p-|\mu|}{2})$.  Moreover, by (\ref{eqn:usualPieri}), $c_{\alpha,\gamma}^{(p)} = 1$. Therefore, 
\begin{equation}
\label{eqn:jdjdj}
N_{\mu,(p),\lambda} = \sum_\beta c_{(\frac{|\mu|+p-|\lambda|}{2}),\beta}^\mu c_{\beta,(\frac{|\lambda|+p-|\mu|}{2})}^\lambda.
\end{equation}
By (\ref{eqn:usualPieri}), $c_{(\frac{|\mu|+p-|\lambda|}{2}),\beta}^\mu \in \{0,1\}$. It is $1$ if and only if one can remove $\frac{|\mu|+p-|\lambda|}{2}$ boxes from different columns of $\mu$ to get $\beta$.  Similarly, $c_{\beta,(\frac{|\lambda|+p-|\mu|}{2})}^\lambda \in \{0,1\}$, and is $1$ if and only if one can add $\frac{|\mu|+p-|\lambda|}{2}$ boxes to different columns of $\beta$ to get $\lambda$.  We are done proving the $N_{\mu,(p),\lambda}$ claim by (\ref{eqn:jdjdj}). The assertion (\ref{eqn:inotherwords}) is a straightforward rephrasing of the first claim.
\end{proof}

\begin{example}
We have
\begin{multline}\nonumber
{s}_{[2,1]}{s}_{[3]} = {s}_{[1,1]}
+ {s}_{[2]}+ {s}_{[2, 1, 1]} 
+ {s}_{[2, 2]}
+ 2{s}_{[3, 1]}
+ {s}_{[4]}
+ {s}_{[3,2,1]}+ {s}_{[4,1,1]}
+{s}_{[4,2]}
+ {s}_{[5,1]}.
\end{multline}
For example, $\lambda=(3,1)$ can be obtained in two ways from $\mu=(2,1)$ using $j=1$:
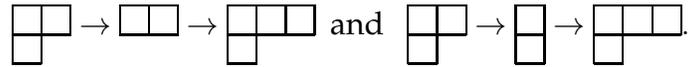
\[\tableau{\ & \ \\  \ }\to \tableau{\ & \ \\ }\to \tableau{\ & \ & \ \\ \ } \text{\ \ and \ \  }
\tableau{\ & \ \\  \ }\to \tableau{\ \\ \  }\to \tableau{\ & \ & \ \\ \ }.\]
This explains the multiplicity in the computation.\qed
\end{example}

Proposition~\ref{thm:Pieri} immediately implies a special case that we also use.

\begin{corollary}
\label{prop:boxcase}
${s}_{[(1)]}   {s}_{[\nu]} =\sum_{\lambda} {s}_{[\lambda]}$,
where the sum is over all partitions $\lambda$
obtained by adding a box to $\nu$ or removing a box from $\nu$.\footnote{Let $({\mathbb Y},\leq)$ be Young's poset. Standard tableaux biject with walks in ${\mathbb Y}$ from $\emptyset$ to $\lambda$, where each step is a covering relation. Iterating (\ref{eqn:usualPieri})
shows $s_{(1)}^k=\sum_{\lambda} f^{\lambda} s_{\lambda}$,
where $f^{\lambda}$ counts standard Young tableaux of shape $\lambda$.
 An \emph{oscillating tableau} of shape $\lambda$ and length $k$ is a walk in ${\mathbb Y}$ starting at $\emptyset$ and ending at $\lambda$ with $k$ edges such that each step
 $\theta\to \pi$ either has $\pi/\theta$ or $\theta/\pi$ being a single box. Let $o^{\lambda,k}$ be the number of these tableaux. It is known that
 $o^{\lambda,k}={k\choose |\nu|}(k-1)!!f^{\nu}$. Iterating Corollary~\ref{prop:boxcase} gives
 $s_{[(1)]}^k=\sum_{\lambda} o^{\lambda,k}s_{[\lambda]}$; see \cite{Berele, Sundaram, Okada}.}
\end{corollary}

\section{Shape of $s_{[\mu]}s_{[\nu]}$}\label{sec:product}

We describe some salient features of $s_{[\mu]}s_{[\nu]}$. Let $\mu \Delta \nu=(\mu\setminus \nu)\cup (\nu\setminus\mu)$ be the symmetric difference of $\lambda$ and $\mu$. 
\begin{theorem}
\label{thm:interval}
Fix $\mu,\nu\in {\sf Par}$. 

\begin{itemize}
\item[(I)] There exists $\lambda\in {\sf Par}$ with $|\lambda|=k$ and
$N_{\mu,\nu,\lambda}>0$ if and only if 
\[k\equiv |\mu\Delta\nu| \ ({\mathrm{mod}}\  2)  \text{\ \ and \ $|\mu\Delta\nu|\leq k \leq |\mu|+|\nu|$.}\]
\item[(II)] If $N_{\mu,\nu,\lambda}>0$ with $|\lambda|>|\mu\Delta\nu|$, there
exists $\lambda^{\downarrow\downarrow}$ such that $N_{\mu,\nu,\lambda^{\downarrow\downarrow}}>0$, $\lambda^{\downarrow\downarrow}\subset\lambda$
and $|\lambda^{\downarrow\downarrow}|=|\lambda|-2$.
\item[(III)] If $N_{\mu,\nu,\lambda}>0$ with $|\lambda|<|\mu|+|\nu|$, there exists
$\lambda^{\uparrow\uparrow}$ such that $N_{\mu,\nu,\lambda^{\uparrow\uparrow}}>0$,
$\lambda \subset \lambda^{\uparrow\uparrow}$ and 
$|\lambda^{\uparrow\uparrow}|=|\lambda|+2$.
\end{itemize}
\end{theorem}
\begin{proof}

(I): By Proposition~\ref{prop:reformulate}, $N_{\mu,\nu,\lambda}>0$ if and only if
there exists $\alpha\subseteq \mu\wedge\nu$
such that 
$\langle s_{\mu/\alpha}s_{\nu/\alpha},s_{\lambda}\rangle >0$. Now, 
\begin{equation}
\label{eqn:March25tyty}
s_{\mu/\alpha}s_{\nu/\alpha}\neq 0 \iff \alpha\subseteq \mu\wedge\nu
\end{equation}
Thus, by (\ref{eqn:skew}) and (\ref{eqn:prodabc}) combined, it suffices to 
characterize the possible values of $\deg(s_{\mu/\alpha}s_{\nu/\alpha})$.
By taking $\alpha=\emptyset$ we obtain that $\deg(s_{\mu/\alpha}s_{\mu/\alpha})\leq |\mu|+|\nu|$. Considering $\alpha=\mu\wedge\nu$  shows
$|\mu\Delta\nu|\leq \deg(s_{\mu/\alpha}s_{\mu/\alpha})$. Also, it is clear that
\begin{equation}
\label{eqn:March25zzz}
\deg(s_{\mu/\alpha}s_{\nu/\alpha})\equiv \deg(s_{\mu/\theta}s_{\nu/\theta}) 
\ ({\mathrm{mod}} \ 2), \ \  \forall \alpha,\theta\subseteq \mu\wedge\nu.
\end{equation}
Thus (I) follows.

(II) We need two claims.

\begin{claim}
\label{claim:updown}
Suppose $c_{\alpha,\beta}^{\mu}> 0$ and $\alpha \subset \alpha^{\uparrow}\subseteq \mu$
with $|\alpha^{\uparrow}/\alpha|=1$. Then there exists $\beta^{\downarrow}\subset\beta$
with $|\beta/\beta^{\downarrow}|=1$ such that $c_{\alpha^{\uparrow}\beta^{\downarrow}}^{\mu}>0$.
\end{claim}
\noindent
\emph{Proof of Claim~\ref{claim:updown}:} It is possible to prove this using
the Littlewood-Richardson rule, however for brevity, we will use a result \cite[Proposition~2.1]{ARY} which concerns the \emph{equivariant} generalization $C_{\lambda,\mu}^{\nu}$ of 
$c_{\lambda,\mu}^{\nu}$. For our purposes, it suffices to know that $C_{\lambda,\mu}^{\nu}$ is a polynomial that is nonzero only if $|\lambda|+|\mu|\geq |\nu|$ and moreover,
$C_{\lambda,\mu}^{\nu}=c_{\lambda,\mu}^{\nu}$ if $|\lambda|+|\mu|=|\nu|$.

Given $c_{\alpha,\beta}^{\mu}>0$, by part (A) of \cite[Proposition~2.1]{ARY}
for any $\alpha\subset \alpha^{\uparrow} \subset \mu$ (where $\alpha^{\uparrow}$ is
$\alpha$ with a box added) we have $C_{\alpha^{\uparrow},\beta}^{\mu}\neq 0$ (as a polynomial). However, by part (B) of \cite[Proposition~2.1]{ARY}, \emph{there exists}
$\beta^{\downarrow}\subset \beta$ (which is $\beta$ with a box removed) such that
$C_{\alpha^{\uparrow},\beta^{\downarrow}}^{\mu}\neq 0$. Since $|\alpha^\uparrow|+|\beta^\downarrow|=|\mu|$, $C_{\alpha^\uparrow,\beta^\downarrow}^{\mu}=
c_{\alpha^\uparrow,\beta^\downarrow}^{\mu}>0$.\qed

\begin{claim}
\label{claim:Apr2abc}
Suppose $\beta,\gamma, \beta^\uparrow,\gamma^\uparrow$ are partitions such that
$\beta\subset \beta^\uparrow$ where $|\beta^\uparrow/\beta|=1$,
and $\gamma\subset\gamma^\uparrow$ where $|\gamma^\uparrow/\gamma|=1$. If
$c_{\beta^\uparrow,\gamma^\uparrow}^{{\overline\lambda}}>0$ then there exists 
${\overline{\lambda}}^{\downarrow\downarrow}\subset {\overline\lambda}$
with $|{\overline\lambda}/{\overline\lambda}^{\downarrow\downarrow}|=2$
such that $c_{\beta,\gamma}^{{\overline\lambda}^{\downarrow\downarrow}}>0$.
\end{claim}
\noindent
\emph{Proof of Claim~\ref{claim:Apr1abc}:} By Pieri's rule (\ref{eqn:usualPieri}), 
\[s_{\beta}s_{(1)}=s_{\beta^{\uparrow}}+\text{(positive sum of Schur functions)}\]
and
\[s_{\gamma}s_{(1)}=s_{\gamma^{\uparrow}}+\text{(positive sum of Schur functions)}.\]
Hence, 
\begin{equation}
\label{eqn:Apr1bbb}
s_{\beta}s_{\gamma}s_{(1)}^2=s_{\beta^{\uparrow}}s_{\gamma^{\uparrow}}+\text{(positive sum of
Schur functions)}
\end{equation}
Expanding the lefthand side of (\ref{eqn:Apr1bbb}) into the basis of Schur functions, gives
\[s_{\beta}s_{\gamma}s_{(1)}^2=\sum_{\theta} c_{\beta,\gamma}^{\theta}(s_{\theta}s_{(1)}^2).\]
Hence, by Pieri's rule (\ref{eqn:usualPieri}), 
\[[s_{\kappa}]s_{\beta}s_{\gamma}s_{(1)}^2\neq 0\] only if $\kappa$ is obtained
from $\theta$ with $c_{\beta,\gamma}^{\theta}>0$ with $\theta\subset \kappa$
and $|\kappa/\theta|=2$. 
Now, since the righthand side of (\ref{eqn:Apr1bbb}) is Schur positive
the same must be true of any $\kappa$ such that $[s_{\kappa}]s_{\beta^\uparrow}s_{\gamma^{\uparrow}}$. In particular this is true of $\kappa={\overline\lambda}$.\qed

Since $N_{\mu,\nu,\lambda}>0$, there exists $(\alpha,\beta,\gamma)$
such that $c_{\alpha,\beta}^{\mu} c_{\alpha,\gamma}^{\nu}c_{\beta,\gamma}^{\lambda}>0$.
Since $|\lambda|>|\mu\Delta\nu|$ we must have $\alpha\subsetneq \mu\wedge\nu$.
Hence let 
\[\alpha\subsetneq \alpha^{\uparrow}\subseteq \mu\wedge \nu\] 
be $\alpha$
with a box added. By two applications of Claim~\ref{claim:updown}, there exists
$\beta^{\downarrow}$ and $\gamma^{\downarrow}$ which are respectively $\beta$
and $\gamma$ with a box removed such that 
$c_{\alpha^{\uparrow}\beta^{\downarrow}}^{\mu}, c_{\alpha^{\uparrow}\gamma^{\downarrow}}^{\nu}>0$. Now apply Claim~\ref{claim:Apr2abc} with
${\overline\lambda}=\lambda$ and $\beta^{\downarrow}, \gamma^{\downarrow},\beta,\gamma$. The conclusion is that $(\alpha^{\uparrow},\beta^{\downarrow},\gamma^{\downarrow})$ is a witness for 
$N_{\mu,\nu,\lambda^{\downarrow\downarrow}}$
and $\lambda^{\downarrow\downarrow}\subset \lambda$ of two smaller size, as desired.

(III): We need two additional claims.

\begin{claim}
\label{claim:Apr1abc}
Suppose $c_{\beta,\gamma}^{\lambda}>0$. If $\gamma^{\uparrow}\supset \gamma$
with $|\gamma^\uparrow/\gamma|=1$ then there exists $\lambda^\uparrow\supset
\lambda$ with $|\lambda^\uparrow/\lambda|=1$ such that $c_{\beta,\gamma^\uparrow}^{\lambda^\uparrow}>0$. 
\end{claim}
\noindent
\emph{Proof of Claim~\ref{claim:Apr1abc}:} Fix a  rectangle $R=\ell\times (m-\ell)$
(for some positive integers $\ell,m$)
sufficiently large to contain $\beta,\gamma,\lambda$. Given a Young diagram $\theta\subseteq R$ let $\theta^{\vee}$ be the $180$-degree rotation of
$R\setminus \theta$. A Schubert calculus symmetry 
for the Grassmannian ${\sf Gr}_{\ell}({\mathbb C}^m)$ states that
\begin{equation}
\label{eqn:Schubsym}
c_{\beta,\gamma}^{\lambda}=
c_{\lambda^{\vee},\gamma}^{\beta^{\vee}}.
\end{equation}
Choose $\ell,m$ sufficiently large  so that $\gamma^{\uparrow}\subset \beta^{\vee}$. By Claim~\ref{claim:updown}, there
exists $(\lambda^{\vee})^{\downarrow}$ which is $\lambda^{\vee}$ with a box
removed such that $c_{(\lambda^{\vee})^{\downarrow},\gamma^{\uparrow}}^{\beta^{\vee}}>0$. By (\ref{eqn:Schubsym}), 
\[0<c_{(\lambda^{\vee})^{\downarrow},\gamma^{\uparrow}}^{\beta^{\vee}}=c_{\beta,\gamma^{\uparrow}}^{((\lambda^{\vee})^{\downarrow})^{\vee}}.\]  
By definition of ``${}^\vee$'',
$((\lambda^{\vee})^{\downarrow})^{\vee}$ is of the form 
$\lambda^{\uparrow}$ such that $c_{\beta,\gamma^{\uparrow}}^{\lambda^{\uparrow}}>0$.\qed

\begin{claim}
\label{claim:Apr1hdhd}
Suppose $c_{\alpha,\beta}^{\mu}>0$. For any $\emptyset \subseteq \alpha^{\downarrow}\subset \alpha$
with $|\alpha/\alpha^{\downarrow}|=1$ there exists $\beta^{\uparrow}\supset \beta$
with $|\beta^{\uparrow}/\beta|=1$ such that $c_{\alpha^{\downarrow},\beta^{\uparrow}}^{\mu}>0$.
\end{claim} 
\noindent
\emph{Proof of Claim~\ref{claim:Apr1hdhd}:} Since $c_{\alpha,\beta}^{\mu}>0$,
there exists a LR tableau $T$ of shape $\mu/\alpha$ and
content $\beta$. We are done once we modify $T$ to give a LR tableau $T'$ of shape
$\mu/\alpha^{\downarrow}$ and content $\beta^{\uparrow}$, as follows: Place $1$ in
$b_1=\alpha/\alpha^{\downarrow}$. Find the first $1$ (if it exists, say in $b_2$) in the column reading (top to bottom, right to left) word order after $b_1$ and turn that into a $2$. Next, find the first
$2$ (again, if it exists, say in $b_3$) in the column reading word order after $b_2$
and change that to a $3$. We terminate and output $T'$ when, after replacing the $k-1$ in $b_k$ with $k$, there is no later $k$ in the column reading order. 

Since the number of boxes of $T$ is finite, this process does end. $T'$ is clearly of the desired shape. The content of $T'$ is 
\[\beta^\uparrow:=(\beta_1,\beta_2,\ldots,
\beta_{k}+1,\beta_{k+1},\ldots).\] 
It remains to check two things:

($T'$ is semistandard): Since $T'(b_1)=1$, we can only violate semistandardness if
the box $d_1$ directly below $b_1$ has $T(d_1)=1$. However, in that case $T'(d_1)=2$,
by construction. In general, since 
\[T'(b_j):=T(b_j)+1(=j)  \text{\ for $2\leq j\leq k$,}\] 
the entry in $b_j$ of $T'$ can only cause a problem with semistandardness with the box $d_j$ directly below,
or the box $r_j$ directly to the right. The former is only a concern if $T(d_j)=j$, but in that
case $T'(d_j)=j+1$. 

The latter concern occurs if $T(r_j)=j-1$. If $b_{j-1}$ is in a column
strictly to the right of $b_j$ then $T(r_j)=j-1$  cannot occur
since the $j-1$ in $r_j$ occurs strictly between $b_{j-1}$ and $b_j$ in the column reading
word. This contradicts the
definition of $b_j$. So we may assume $b_{j-1}$ is in the same column as $b_j$.
Since 
\[T(b_{j-1}):=j-2 \text{\ and $T(b_j):=j-1$},\] in fact,
$b_{j-1}$ is immediately above $b_{j}$, i.e., $d_{j-1}=b_j$. Since we assume
$T(r_j)=j-1$, semistandardness of $T$ implies $T(r_{j-1})=j-2$, which by the same argument implies
$b_{j-2}$ is directly above $b_{j-1}$ (otherwise we would contradict the definition of
$b_{j-1}$. Repeating this logic tells us that $b_2,b_3,\ldots,b_{j}$
are consecutive boxes in the same column with $T(b_2):=1$ and $T(r_2)=1$. However,
this forces $b_1$ to be in a column strictly right of $b_2$. Since $T(r_2)=1$ and
$r_2$ is between $b_1$ and $b_2$, we contradict the definition of $b_2$. Thus, the
situation $T(r_j)=j-1$ of this paragraph cannot actually occur.

($T'$ is ballot):
It is well-known that any semistandard tableau is ballot with respect to the row reading if
and only if it is ballot with respect to the column reading. For $j\geq 2$, we need to show that $T'$ is \emph{$(j-1,j)$-ballot}, that is, the number of $j-1$'s appearing at any given point of the column reading word exceeds the number of $j$'s
at the same point. If $j>k+1$ then the $j-1$'s and $j$'s in $T'$ and $T$ are in the exact same
positions, and $T'$ is $(j-1,j)$-ballot since $T$ is. If $j=k+1$ the same is true except
$T'$ has an additional $j-1=k$ at $b_k$, and ballotness similarly follows.

Now suppose $j\leq k$.
The only boxes $b_t$ ($1\leq t\leq k$) that contain $j-1$ or $j$ in $T$ or $T'$ are $b_{j-1},b_j$ and $b_{j+1}$.
Hence consider four regions of $T'$:
(i) strictly before $b_{j-1}$; (ii) starting from $b_{j-1}$ to before $b_j$; (iii)
starting from $b_j$ until before $b_{j+1}$; and (iv) $b_{j+1}$ and thereafter 
(in the column reading
order).  Below, let $w[b]$ be a partial reading word of $T$ that ends at a box $b$. Let $w'[b]$ be the
word using the same boxes of $T'$. 

In region (i), the $j$'s and $(j-1)$'s are in the same positions in both $T$ and $T'$.
Hence since $w[b]$ is $(j-1,j)$-ballot, the same is true of $w'[b]$ for any $b$ in (i). For any
$b$ in (ii), $w'[b]$ has one
more $j-1$ than $w[b]$ (since $T(b_{j-1})=j-2$ and $T'(b_{j-1})=j-1$. Hence, $w'[b]$ is
$(j-1,j)$-ballot because this is true of $w[b]$. 

For any $b$ in region (iii), $w'[b]$ and $w[b]$ have the same number of $(j-1)$'s
but $w'$ has one more $j$. There are two cases. 

\noindent
\emph{Case 1: ($b_{j+1}$ exists, \emph{i.e.}, $j<k$ and region  (iv) exists)} If $w'[b]$ is not $(j-1,j)$-ballot, then it follows $w[b_{j+1}]$ is not $(j-1,j)$-ballot,
a contradiction. Finally, if $b$ is in (iv), $w[b]$ and $w[b']$ have the same number of
$(j-1)$'s and $j$'s, so we are again done. 

\noindent
\emph{Case 2: ($b_{j+1}$ does not exist, \emph{i.e.}, $j=k$ and region (iv) does not exist)} This case means there are no
$j$'s in $T$ after $b_j$. Hence if $w'[b]$ fails to be $(j-1,j)$-ballot for any $b$ weakly
after $b_j$, in fact $w'[b_j]$ is not $(j-1,j)$-ballot. By definition, $w[b_j]$ has the same number of $(j-1)'s$ but one less
$j$ than $w'[b_j]$. Since $w'[b_j]$ is not $(j-1,j)$-ballot, it must be that $w[b_j]$ has the same number of $(j-1)$'s and $j$'s. Let $b^\circ$ be the box immediately before
$b_j$ in the reading order. Since $T(b_j)=j-1$ we conclude $w[b^\circ]$ is not $(j-1,j)$-ballot,
a contradiction. \qed

Since $N_{\mu,\nu,\lambda}>0$ there exists $(\alpha,\beta,\gamma)$ such that
$c_{\alpha,\beta}^{\mu}c_{\alpha,\gamma}^{\nu}c_{\beta,\gamma}^{\lambda}>0$.
Remove any corner from $\alpha$ to obtain $\alpha^{\downarrow}$. By two applications of Claim~\ref{claim:Apr1hdhd}  there exists $\beta^\uparrow$ and $\gamma^\uparrow$
such that $c_{\alpha^\downarrow,\beta^\uparrow}^\mu, c_{\alpha^\downarrow,\gamma^\uparrow}^{\nu}>0$. By two applications of Claim~\ref{claim:Apr1abc},
there exists $\lambda^{\uparrow\uparrow}$ (as in the theorem statement) such that
$c_{\beta^\uparrow,\gamma^\uparrow}^{\lambda^{\uparrow\uparrow}}>0$. 
Hence
$(\alpha^\downarrow,\beta^\uparrow,\gamma^\uparrow)$ witnesses that
$N_{\mu,\nu,\lambda^{\uparrow\uparrow}}>0$.
\end{proof}

\begin{example}
If $\mu=(3)$ and $\nu=(2,1)$ then $|\mu\Delta\nu|=2$ and 
$|\mu|+|\nu|=6$. We compute:
\begin{multline}\nonumber
{s}_{[3]}{s}_{[2,1]}=
{s}_{[1,1]}+{s}_{[2]}+{s}_{[2,1,1]} +
{s}_{[2,2]}+ 2{s}_{[3,1]}+ {s}_{[4]}
+ {s}_{[3,2,1]}+ {s}_{[4,1,1]}+{s}_{[4,2]}+
{s}_{[5,1]}.
\end{multline}
The reader can check agreement with Theorem~\ref{thm:interval}.\qed
\end{example}

There seems to be another ``structural'' aspect of (\ref{eqn:polyver}). Define 
\[h^{\mu,\nu}_{t}=\sum_{\lambda: |\lambda|=|\mu\Delta\nu|+2t} N_{\mu,\nu,\lambda}.\]

A sequence $(a_k)_{k=0}^{N}$ is \emph{unimodal} if  
there exists $0\leq m\leq N$ such that 
\[0\leq a_0\leq a_1\leq \ldots\leq  a_m \geq a_{m+1}
\geq \ldots a_{N-1}\geq a_N.\]
\begin{conjecture}[Unimodality]
\label{conj:log}
The sequence $\{h^{\mu,\nu}_{t}\}_{t=0}^{|\mu\wedge \nu|}$ is a unimodal sequence.
\end{conjecture}
We checked Conjecture~\ref{conj:log} for all $s_{[\mu]}s_{[\nu]}$ where
$0\leq |\mu|,|\nu|\leq 7$, and many larger cases. Theorem~\ref{thm:interval} (II) and (III) suggest proving Conjecture~\ref{conj:log} by constructing chains in Young's poset, each element $\lambda$ appearing $N_{\mu,\nu,\lambda}$-many times, ``centered'' at $m$:

\begin{example}
Continuing the previous example, $\{h^{\mu,\nu}_{t}\}_{t=0}^{3}=2,5,4$. Here $m=1$
and we are suggesting that the following chains demonstrate the unimodality:
\begin{align*}
(1,1)\subset & (2,2)  \subset (4,2)\\
(2)\subset & (2,1,1)  \subset (4,1,1)\\
\ & (3,1)  \\
\ & (3,1) \subset (3,2,1)\\
\ & (4)  \subset (5,1)
\end{align*}
There is choice in the chains; in the first and third chains we could 
interchange the roles of $(2,2)$ and $(3,1)$.\qed
\end{example}

A sequence is \emph{log-concave} if 
\[a_{t}^2\geq a_{t-1} a_{t+1} \text{ \ for $0< t< N$.}\]
Log-concavity implies unimodality. Thus, a warning against Conjecture~\ref{conj:log} is this:

\begin{example}[Log-concavity counterexample]
$\{h^{(2,2),(2,2)}_t\}_{t=0}^{4}=1, 2, 6, 8, 6$ is unimodal but not
log-concave.\qed
\end{example}

%
%
%
%
%
%
%
%
%
%
%

\section{H.~Hahn's notion of detection}
\label{sec:Hahn}
Our study of ${\sf NL}_n$ was stimulated by work of H.~Hahn \cite{Hahnearlier, Hahn}.  Suppose ${\sf H}$ is an irreducible reductive subgroup of ${\sf GL}_N$. H.~Hahn \cite{Hahnearlier} defines that a representation 
\begin{equation}
\label{eqn:detects}
\rho:{\sf GL}_N\to {\sf GL}(V)
\end{equation}
 \emph{detects}
 ${\sf H}$ if ${\sf H}$ stabilizes a line in $V$. She initiates a study of detection, motivated by R.~Langlands' \emph{beyond endoscopy}
proposal \cite{Beyond} towards proving his functoriality conjecture \cite{Langlands} (see \cite{Hahnearlier, Hahn} for elucidation and further references). 

The general question stated in \cite{Hahnearlier} is to determine which algebraic subgroups of ${\sf GL}_N$ are detected by a representation (\ref{eqn:detects})? In \cite{Hahn}, this
question is studied using the classical groups 
${\sf G}={\sf SO}_{2n+1}, {\sf Sp}_{2n}, {\sf SO}_{2n}$
 (where in the latter case $n$ is assumed to be even) and where 
 $\rho:{\sf GL}_N\to {\sf GL}_{N^3} \text{\ is $\rho= \otimes^3$}$, 
 i.e., the corresponding ${\sf GL}_N$-module is ${\mathbb C}^N \otimes {\mathbb C}^N  \otimes {\mathbb C}^N$ with the diagonal (standard) action of ${\sf GL}_N$ where 
 $g\cdot(u  \otimes v  \otimes w)=gu \otimes gv \otimes gw$.

In each case, H.~Hahn considers the (irreducible) ${\sf G}$-module
 ${\mathbb S}_{[\lambda]}(W)$ from the introduction (in type $D_n$ she assumes $\lambda_n=0$).
If $r:{\sf G}\to {\sf GL}_N$ is the ${\sf G}$-representation corresponding to 
${\mathbb S}_{[\lambda]}(W)$, then it makes sense to define ${\sf H}$ as the Zariski closure of
$r({\sf G})$ inside ${\sf GL}_N$. That is, in the notation of \cite{Hahn},
${\sf H}$ 
is the irreducible subgroup of
${\sf GL}_N$ of interest.

 Theorem~1.5 of \emph{ibid}.~proves that if $|\lambda|$ is odd then $\rho= \otimes^3$
does not detect ${\mathbb S}_{[\lambda]}(W)$. Conversely, 
when $|\lambda|$ is even. Theorem~1.6 of \emph{ibid.}~gives
three infinite subfamilies of ${\sf Par}_n$ where $\rho= \otimes^3$ detects 
${\mathbb S}_{[\lambda]}(W)$. 

We give a short proof of a complete converse. 

\begin{theorem}
\label{thm:main}
Let $\lambda\in {\sf Par}_n$. Then $\rho= \otimes^3$
detects ${\mathbb S}_{[\lambda]}(W)$ if $|\lambda|\equiv 0 \ ({\mathrm{mod}}\  2)$.\footnote{One might compare this parity characterization to \cite[Theorem~1.5]{Hahn} which shows that
$G:={\rm Sym}^{n-1}({\sf SL}_2)\hookrightarrow {\sf GL}_n$ is detected by $\rho:={\rm Sym}^3$ if and only if $n\equiv 1 \ ({\mathrm{mod}}\  4)$.}
\end{theorem}

\noindent
\emph{Proof of Theorem~\ref{thm:main}:} 
 Hahn's \cite[Proposition~3.1]{Hahn} shows
that 
\begin{equation}
\label{eqn:triplelambdas}
\rho= \otimes^3 \text{ \ detects ${\mathbb S}_{[\lambda]}(W)$ if and only if
$N_{\lambda,\lambda,\lambda}>0$.}
\end{equation}
In \emph{ibid.} this is used to prove $(\Rightarrow)$.\footnote{This follows from Proposition~\ref{prop:easy}(IV), which just extends the argument made
in \cite{Hahn}.}
Therefore, (\ref{eqn:Newell-Littlewood}) shows
\begin{lemma}
\label{thelemma}
$\rho= \otimes^3$ detects ${\mathbb S}_{[\lambda]}(W)$ if there exists $\mu\in {\sf Par}_n$ such that $c^{\lambda}_{\mu,\mu}>0$.
\end{lemma}

\begin{claim}
\label{theclaim}
For any $\lambda\in {\sf Par}_n$ with $|\lambda| = 2m$, there exists 
$\mu\in {\sf Par}_n$ such that $c^{\lambda}_{\mu,\mu} > 0$.
\end{claim}
\noindent\emph{Proof of Claim~\ref{theclaim}:}
Since $|\lambda|$ is even, there are an even number of odd parts in $\lambda$. Let 
\[\lambda_{i_{1}} \geq \ldots \geq \lambda_{i_{2k}}\] 
be the odd parts of $\lambda$.  

Define $\mu = (\mu_{1},\mu_{2},\ldots,\mu_{n})$ to be a partition of $m$, where 
\[ \mu_{j} = \begin{cases} \frac{\lambda_{j}}{2} &  \lambda_{j} $ is even$ \\ \frac{\lambda_{j}+1}{2} & \lambda_j $ is odd and $ j \leq i_{k} \\ \frac{\lambda_{j}-1}{2} & \lambda_j $ is odd and $ j > i_{k}\end{cases}\]
We show $c^{\lambda}_{\mu,\mu}>0$ by giving an explicit
ballot filling of $\lambda/\mu$ with content $\mu$ (see~Section~\ref{sec:LR}).

For $\lambda_i$ even, fill in the rightmost $\frac{\lambda_i}{2}$ boxes with $i$. 
For a row $i_j$ of $\lambda$ with an odd number of boxes, fill in the rightmost $\frac{\lambda_{i_j}-1}{2}$ boxes in the row  with $i_j$. There are $\frac{\lambda_{i_j}-1}{2}$ boxes in each of the top $k$ rows with odd parts. Hence those boxes are entirely filled. There are
$\frac{\lambda_{i_j}+1}{2}$ boxes in each of the bottom $k$ rows of odd parts. For these rows, one
box remains unfilled by the above step. Fill in the empty box in row $i_{k+j}$ with $i_j$; for the purposes of discussion below, we will call this box \emph{extraordinary}. 
It will also be convenient to call indices $j$ \emph{$\lambda$-even} if $\lambda_j$ is even,
\emph{$\lambda$-top-odd} if $\lambda_{j}$ is odd and $j\leq i_k$, and $\emph{$\lambda$-bottom-odd}$
otherwise.
Let $T$ be this filling. (See Example~\ref{exa:April16abc} below.) We must check three things:

\noindent
($T$ is semistandard): By construction, $T$ is row-semistandard. It remains to show column strictness. This is clear
  when comparing adjacent rows $j$ and $j+1$ that are either $\lambda$-even, or $\lambda$-top-odd,
 since those only use those labels
 in their respective rows. If either row is bottom-odd, notice that any extraordinary box is either directly beneath an empty  square or another extraordinary box. Since extraordinary boxes are labeled
 in strictly increasing  from top to bottom, we are done.
 
\noindent
($T$ has content $\mu$): If $j$ is $\lambda$-even, then 
$\mu_j=\frac{\lambda_j}{2}$ and there are that many $j$'s in row $j$ of $T$ (and nowhere else). Otherwise, if $j$ is $\lambda$-top-odd then we are deficient one label of $j$ in that row. By construction, this missing $j$ appears in row $i_{k+j}$.

\noindent
($T$ is ballot):  If $j$ is $\lambda$-even, the ballotness holds since all $j$'s appear
in row $j$ and all $j+1$'s appear in the row $j+1$ or further south, and since
$\mu_j\geq \mu_{j+1}$. Next, suppose $j+1$ (but not $j$) is $\lambda$-even. Hence $\lambda_{j+1}<\lambda_j$ and row $j$
of $T$ will contain 
$\frac{\lambda_j-1}{2}\geq \frac{\lambda_{j+1}}{2}$ many $j$'s; these
$j$'s will be read before the $\frac{\lambda_{j+1}}{2}$-many $j+1$'s of $T$, which appear
only in row $j+1$. Similarly, we are done if $j$ and $j+1$ are both
$\lambda$-bottom-odd, or (since extraordinary boxes' labels increase top-down) if
both are $\lambda$-top-odd. Finally, say
$j$ is $\lambda$-top-odd, $j+1$ is $\lambda$-bottom-odd.
Then row $j$ of $T$ has $\frac{\lambda_j-1}{2}$ many $j$'s and 
all $\frac{\lambda_{j+1}-1}{2}(\leq \frac{\lambda_{j}-1}{2})$ many $j+1$'s appear in row $j+1$ of $T$, so
ballotness follows.
\qed

In view of Lemma~\ref{thelemma}, Claim~\ref{theclaim} completes the proof of the
theorem.\qed

\begin{example}\label{exa:April16abc}
To illustrate the proof of Claim~\ref{theclaim}, let 
$\lambda=(14,11,10,8,8,7,6,6,5,5,4,3,2,1)$. 
Hence
$2k=6$, $(i_1,i_2,i_3,i_4,i_5,i_6)=(2,6,9,10,12,14)$, and 
$\mu=(7,6,5,4,4,4,3,3,3,2,2,1,1,0)$. 
In this case,
$T$ is
\[\tableau{X& X & X & X& X& X& X& 1& 1&1&1&1 & 1 &1 \\
X&X&X&X&X&X&2&2&2&2&2\\
X&X&X&X&X&3&3&3&3&3\\
X&X&X&X&4&4&4&4\\
X&X&X&X&5&5&5&5\\
X&X&X&X&6&6&6\\
X&X&X&7&7&7\\
X&X&X&8&8&8\\
X&X&X&9&9\\
X&X&{\bf 2}&10&10\\
X&X&11&11\\
X&{\bf 6}&12\\
X&13\\
{\bf 9}
}\]
where we have boldfaced the labels in the exceptional boxes.
\qed
\end{example}

Given a partition $\lambda=(\lambda_1, \lambda_2, \ldots)$ let $k\lambda=(k\lambda_1, k\lambda_2,\ldots)$. Theorem~\ref{thm:main} combined with ({\ref{eqn:triplelambdas})
implies:

\begin{corollary}
\label{cor:triplelambdas}
If $|\lambda|\equiv 0 \ ({\mathrm{mod}}\  2)$ then $N_{\lambda,\lambda,\lambda}>0 \iff
N_{k\lambda, k\lambda, k\lambda}>0$ for all $k\in {\mathbb Z}_{\geq 1}$.
\end{corollary}

The simplicity of this ``saturation'' statement suggested the ideas of the next
section.

\section{Polytopal results}
\label{sec:polytope}

\subsection{Newell-Littlewood polytopes} Fix $\lambda,\mu,\nu\in {\sf Par}_n$. Let $a_{i}^j, b_i^j, \gamma_{i}^j \in {\mathbb R}$
for $1\leq i,j\leq n$ and consider the linear constraints:

\begin{enumerate}
\item \emph{Non-negativity}: For all $1\leq i,j\leq n$, $\alpha_{i}^j,\beta_i^j,\gamma_i^j\geq 0$
\item \emph{Shape constraints}: For all $k$,
\begin{enumerate}
\item
$\sum_j \alpha_k^j + \sum_i \beta_i^k = \mu_k$
\item
$\sum_j \gamma_k^j + \sum_i \alpha_i^k = \nu_k$
\item
$\sum_j \beta_k^j + \sum_i \gamma_i^k = \lambda_k$
\end{enumerate}
\item \emph{Tableau/semistandardness constraints}: For all $k,l$:
\begin{enumerate}
\item
$\sum_{j} \alpha_{k+1}^j + \sum_{i\leq l} \beta_i^{k+1} \leq \sum_{j} \alpha_k^j + \sum_{i<l} \beta_i^k$
\item
$\sum_{j} \gamma_{k+1}^j + \sum_{i\leq l} \alpha_i^{k+1} \leq \sum_{j} \gamma_k^j + \sum_{i<l} \alpha_i^k$
\item
$\sum_{j} \beta_{k+1}^j + \sum_{i\leq l} \gamma_i^{k+1} \leq \sum_{j} \beta_k^j + \sum_{i<l} \gamma_i^k$
\end{enumerate}
\item \emph{Ballot constraints}: For all $k,l$:
\begin{enumerate}
\item
$\sum_{i<k} \alpha_{l}^i \geq \sum_{i\leq k} \alpha_{l+1}^{i}$
\item
$\sum_{i<k} \beta_{l}^i \geq \sum_{i\leq k} \beta_{l+1}^{i}$
\item
$\sum_{i<k} \gamma_{l}^i \geq \sum_{i\leq k} \gamma_{l+1}^{i}$
\end{enumerate}
\end{enumerate}

We define the \emph{Newell-Littlewood polytope} in ${\mathbb R}^{3n^2}$ by
\[{\mathcal P}_{\mu,\nu,\lambda}=\{(\alpha_i^j,\beta_i^j,\gamma_i^j)\in {\mathbb R}^{3n^2}: \text{(1)-(4) hold}\}.\]

\begin{theorem}
\label{thm:thepolytope}
$N_{\mu,\nu,\lambda}=\#({\mathcal P}_{\mu,\nu,\lambda}\cap {\mathbb Z}^{3n^2})$.
\end{theorem}
\begin{proof}
By definition, $N_{\mu,\nu,\lambda}$ is the number of LR tableaux $T,U$ and $V$ of shape $\mu/\alpha$, $\nu/\gamma$ and $\lambda/\beta$ respectively, and of
content $\beta$, $\alpha$, and $\gamma$ respectively for any choice of $\alpha$, $\beta$, and $\gamma$ in ${\sf Par}_n$. Given such a triple $(T,U,V)$   
let $\beta_{i}^j$  be the number of $i$'s in the $j^{th}$ row of the ballot filling of $T$.
Similarly, $\alpha_i^j$ and $\gamma_i^j$ are defined with respect to $U$ and $V$ respectively. It is straightforward that
$(\alpha_i^j,\beta_i^j,\gamma_i^j)$ satisfies (1)-(4). 

Conversely, suppose we are given  $(\alpha_i^j,\beta_i^j,\gamma_i^j)\in {\mathcal P}_{\mu,\nu,\lambda}$.
For $1\leq i\leq n$, let
\[\alpha_i:=\sum_{j}\alpha_i^j, \ \beta_i:=\sum_j \beta_i^j,\ \text{and \ } \gamma_i:=\sum_j \gamma_i^j.\]
Notice $\alpha:=(\alpha_1,\ldots,\alpha_n)\in {\sf Par}_n$ by 4(a). Similarly we define $\beta,\gamma\in {\sf Par}_n$.
Now construct $T$ by placing $\beta_i^j$ many $i$'s in row $j$ (indented by $\alpha_i$ many boxes), 
and order the labels in the row to be in increasing from left to right. By 2(a), $T$ is of skew shape $\mu/\alpha$. 
Conditions 3(a) and 4(b) guarantee that $T$ is an LR tableau. 
In the same way, we construct appropriate LR tableaux $U$ and $V$ using $\alpha_i^j, \gamma_i^j$ 
and $\beta$.
This correspondence $(T,U,V)\leftrightarrow (\alpha_i^j,\beta_i^j,\gamma_i^j)$ is
clearly bijective.
\end{proof}

That $N_{\lambda,\mu,\nu}$ counts lattice points in a polytope also follows from
work of A.~Berenstein-A.~Zelevinsky \cite[Section~2.2]{BZ} on the more general tensor product multiplicities, together with \cite[Corollary~2.5.3]{Koike}. Their polytopes are described in terms of root-system datum.
The above gives an \emph{ab initio} approach, similar to one seen in a preprint version of \cite{Mulmuley} for the
Littlewood-Richardson coefficients.

\subsection{Newell-Littlewood semigroups}
The \emph{Littlewood-Richardson semigroup} is
\[{\sf LR}_n=\{(\mu,\nu,\lambda)\in {\sf Par}_n^3:c_{\mu,\nu}^{\lambda}>0\};\]
see, e.g., \cite{Zelevinsky}. 
We define the \emph{Newell-Littlewood semigroup} by
\[{\sf NL}_n=\{(\mu,\nu,\lambda)\in {\sf Par}_n^3: N_{\mu,\nu,\lambda}>0\}.\] 

\begin{corollary}
\label{prop:NLsemigroup}
${\sf NL}_n$ is a semigroup. 
${\sf LR}_n$ is a subsemigroup of ${\sf NL}_n$.
\end{corollary}
\begin{proof}
Suppose 
$(\mu,\nu,\lambda)$ and $(\overline{\mu},\overline{\nu},\overline{\lambda})\in 
{\sf NL}_n$. 
By Theorem~\ref{thm:thepolytope}, there exists a lattice points 
\[(\alpha_i^j,\beta_i^j,\gamma_i^j)\in {\mathcal P}_{\mu,\nu,\lambda}
\text{ \ and $(\overline{\alpha}_i^j,\overline{\beta}_i^j,\overline{\gamma}_i^j)\in
{\mathcal P}_{\overline{\mu},\overline{\nu},\overline{\lambda}}$}.\]
Observe \[(\alpha_i^j,\beta_i^j,\gamma_i^j)+
(\overline{\alpha}_i^j,\overline{\beta}_i^j,\overline{\gamma}_i^j)\in 
{\mathcal P}_{\mu+\overline{\mu},\nu+\overline{\nu},\lambda+\overline{\lambda}; n}\]
is a lattice point. By Theorem~\ref{thm:thepolytope}, 
$N_{\mu+\overline{\mu},\nu+\overline{\nu},\lambda+\overline{\lambda}}>0$ and so
$(\mu+\overline{\mu},\nu+\overline{\nu},\lambda+\overline{\lambda})\in {\sf NL}_n$. Hence ${\sf NL}_n$
is a semigroup.

The remaining assertion follows from Lemma~\ref{prop:easy}(II).
\end{proof}

In turn, Corollary~\ref{prop:NLsemigroup} immediately implies

\begin{corollary}\label{cor:easydirection}
If $N_{\mu,\nu,\lambda}> 0$ then $N_{k\mu, k\nu, k\lambda} > 0$ for every 
$k\geq 1$.
\end{corollary}

A.~Knutson-T.~Tao \cite{KT} established the \emph{saturation property} of $c_{\mu,\nu}^{\lambda}$. That is 
\begin{equation}
\label{eqn:usualsat}
c_{\mu,\nu}^{\lambda}>0 \iff c_{k\mu,k\nu}^{k\lambda}>0,  \ \  \forall k\in {\mathbb Z}_{\geq 1}.
\end{equation}

\begin{conjecture}[Newell-Littlewood Saturation I]
\label{theconj}
Suppose $\lambda,\mu,\nu\in {\sf Par}$ such that $|\lambda|+|\mu|+|\nu|\equiv 0 \ (\!\!\!\!\mod 2)$. If 
$N_{k\mu, k\nu, k\lambda} > 0$ for some $k\geq 1$ then
$N_{\mu, \nu, \lambda}> 0$.
\end{conjecture} 

We checked Conjecture~\ref{theconj} exhaustively for $\lambda,\mu,\nu$ with
$1\leq |\lambda|,|\mu|,|\nu|\leq 8$ and $k=2,3$ as well as many other examples.
The necessity of the parity hypothesis is Lemma~\ref{prop:easy}(V).

This is an \emph{a priori} stronger version of Conjecture~\ref{theconj}:

\begin{conjecture}[Newell-Littlewood Saturation II]
\label{theconj'}
Under the hypotheses of Conjecture~\ref{theconj},
if $N_{k\mu,k\nu, k\lambda}>0$ then there exists $\alpha,\beta,\gamma\in {\sf Par}$
such that $c^{k\mu}_{k\alpha,k\beta}c^{k\nu}_{k\alpha,k\gamma}c^{k\lambda}_{k\beta,k\gamma}>0$.
\end{conjecture}

\begin{proposition}
Conjectures~\ref{theconj} and~\ref{theconj'} are equivalent.
\end{proposition}
\begin{proof}
($\Rightarrow$) Suppose $|\lambda|+|\mu|+|\nu|\equiv 0 \ (\!\!\!\mod 2)$
 and 
$N_{k\mu,k\nu, k\lambda}>0$. By Conjecture~\ref{theconj}, $N_{\mu,\nu,\lambda}> 0$. Hence by (\ref{eqn:Newell-Littlewood}) there exists $\alpha,\beta,\gamma$
such that $c_{\alpha,\beta}^{\mu}, c_{\alpha,\gamma}^{\nu}, c_{\beta,\gamma}^{\lambda}$
are all nonzero. By the semigroup property for Littlewood-Richardson coefficients (Corollary~\ref{prop:NLsemigroup}),
$c_{k\alpha,k\beta}^{k\mu}, c_{k\alpha,k\gamma}^{k\nu}, c_{k\beta,k\gamma}^{k\lambda}$
are also nonzero, as asserted by Conjecture~\ref{theconj'}.

($\Leftarrow$) This holds by (\ref{eqn:Newell-Littlewood}) and 
saturation of the Littlewood-Richardson coefficients (\ref{eqn:usualsat}).
\end{proof}

There has been significant interest in the saturation problem for tensor products of
irreducibles for complex semisimple algebraic groups. Suppose $\mu,\nu,\lambda$ of dominant weights and corresponding irreducibles $V_{\mu}$, $V_{\nu}$
\and $V_{\lambda}$. Let
\[V_{\mu}\otimes  V_{\nu}=\bigoplus_{\lambda} V_{\lambda}^{\oplus m_{\mu,\nu}^{\lambda}}.\]
The aformentioned problem is, if we assume $\mu+\nu-\lambda$ is in the root lattice, is 
\[m_{\mu,\nu}^{\lambda}\neq 0 \iff m_{k\mu,k\nu}^{k\lambda}\neq 0, \ \ \forall k\geq 1?\]

In type $A$, $m_{\mu,\nu}^{\lambda}$ is a Littlewood-Richardson coefficient, and (\ref{eqn:usualsat}) provides an affirmative answer. The answer is negative
for types $B$ and $C$, and is conjectured to be true for all simply-laced types, and
in particular, type $D$. The state of the art is that the type $D$ conjecture is proved for type $D_4$ by
M.~Kapovich-S.~Kumar-J.~J.~Milson \cite{KKM} and more recently by J.~Kiers for 
$D_5,D_6$
\cite{Kiers} (which we refer to for more references). 

Conjecture~\ref{theconj} suggests
that saturation should hold in types $B$ and $C$ at least in the stable range and
under the parity hypothesis. In view of \cite[Theorem~2.3.4]{Koike}, the $D_n$ conjecture should imply Conjecture~\ref{theconj} (taking into account the parity vs root-lattice hypotheses); we thank J.~Kiers for pointing this out (private communication). 
We emphasize that Conjecture~\ref{theconj'}  permits a different approach than \cite{KKM,Kiers} for the
cases at hand. For example, in addition to the infinite family of cases provided by Corollary~\ref{cor:triplelambdas}, we
have:
\begin{theorem}\label{thm:satspecial}
Conjecture~\ref{theconj'} is true if one of $\lambda,\mu,\nu$ is a single row or 
a single column.
\end{theorem}

\begin{proof}
Suppose one of $\lambda,\mu,\nu$ is a single column.
By Lemma~\ref{prop:easy}(I), we may suppose $\mu = (1^t)$.
By assumption, there exists $\alpha,\beta,\gamma$ such that $c^{(k^t)}_{\alpha,\beta},c^{k\nu}_{\alpha,\gamma},c^{k\lambda}_{\beta,\gamma} > 0$.  For convenience, let $[\lambda/\mu]_i$ be the number of boxes
of the $i$-th row of the skew shape $\lambda/\mu$.

\begin{lemma}
\label{prop:rowbound}
If $c^{\lambda}_{\mu,\nu}>0$, then $[\lambda/ \mu]_i \leq \nu_1$ for all $i$.
\end{lemma}
\noindent\emph{Proof of Lemma~\ref{prop:rowbound}:}
Since $c^{\lambda'}_{\mu',\nu'} = c^{\lambda}_{\mu,\nu} > 0$, there is a LR tableau $T$ of $\lambda'/ \mu'$ of content $\nu'$.  The labels of boxes in a given column $C$ of $T$ are distinct. Hence $\#C\leq \ell(\nu')$ and the lemma follows.\qed

The fact $c_{\alpha,\beta}^{(k^t)}>0$ implies that $\alpha,\beta \subseteq (k^t)$ and
hence $\alpha_1,\beta_1 \leq k$. So by Lemma \ref{prop:rowbound},
\begin{equation}
\label{eqn:May12bbb}
[(k\lambda)/ \gamma]_i, [(k\nu)/ \gamma]_i\leq k, \forall i.
\end{equation}
Since $\gamma\subseteq k\nu\wedge k\lambda$, by (\ref{eqn:May12bbb}), for all $i$:
\begin{equation}
\label{eqn:May12aaa}
[(k\lambda)/ (k\nu \wedge k\lambda)]_i\leq [(k\lambda)/\gamma]_i\leq k, 
\text{ \ and  \ } [(k\nu)/(k\nu \wedge k\lambda)]_i\leq [(k\lambda)/\gamma]_i\leq k.
\end{equation}
Also, (\ref{eqn:May12aaa}) and $k\nu \wedge k\lambda = k(\nu \wedge \lambda)$ combined imply
\[ [\lambda/ (\nu \wedge \lambda)]_i, [\nu/(\nu \wedge \lambda)]_i\leq 1, \forall i;\]
that is,
\begin{equation}
\label{May4sss}
|\nu_i-\lambda_i|\leq 1.
\end{equation}
 By Theorem \ref{thm:interval} (I), 
 \[k|\nu \Delta \lambda| = |k\nu \Delta k\lambda|\leq |(k^t)| = kt,\] 
 and so $|\nu \Delta \lambda| \leq t$.  Since
$|\nu\Delta\lambda|\equiv |\nu|+|\lambda| \ (\!\!\!\!\mod 2)$ and (by hypothesis) 
\[|\nu|+|\lambda|+|(1^t)|=|\nu|+|\lambda|+t\equiv 0 \ (\!\!\!\!\mod 2),\] 
we have that $\frac{t - |\nu \Delta \lambda|}{2}\in {\mathbb Z}_{\geq 0}$.

\begin{claim}
\label{May4rrr}
There are at least $\frac{t - |\nu \Delta \lambda|}{2}$ indices $i$ such that $\nu_i = \lambda_i >0$.
\end{claim}
\noindent
\emph{Proof of Claim~\ref{May4rrr}:}
By definition of $\alpha$, $\beta$, and $\gamma$, 
\begin{align*}
kt &= |\alpha| + |\beta| \\
& = |(k\nu)/ \gamma| + |(k\lambda)/ \gamma|\\
&= |(k\nu)/ (k\nu\wedge k\lambda)| + |(k\nu\wedge k\lambda)/ \gamma| + |(k\lambda)/(k\nu\wedge k\lambda)| + |(k\nu\wedge k\lambda)/ \gamma|\\
&= |k\nu \Delta k\lambda| + 2 |(k\nu\wedge k\lambda)/ \gamma|.
\end{align*}

This is equivalent to 
\begin{equation}
\label{eqn:May5aaa}
k\left(\frac{t - |\nu \Delta \lambda|}{2}\right) = |(k\nu\wedge k\lambda)/ \gamma|.
\end{equation}

By (\ref{eqn:May12bbb}), 
\[[(k\nu\wedge k\lambda)/ \gamma]_i\leq [(k\nu)/\gamma]_i\leq k, \forall i.\]
Thus (\ref{eqn:May5aaa}) and the Pigeonhole Principle shows 
\begin{equation}
\label{eqn:May12xyy}
\#\{i: [(k\nu\wedge k\lambda)/ \gamma]_i>0\} \geq \frac{t - |\nu \Delta \lambda|}{2}.
\end{equation}
By (\ref{May4sss}), if $\nu_j\not= \lambda_j$ then $[k\nu\Delta k\lambda]_j=k$.  By (\ref{eqn:May12bbb}),
 $k\nu_j-\gamma_j, k\lambda_j-\gamma_j\leq k$.
Hence 
\begin{multline}
k \geq \max\{k\nu_j,k \lambda_j\} - \gamma_j = (\max\{k \nu_j,k \lambda_j\} -\min\{k \nu_j,k \lambda_j\})\\
+(\min\{k \nu_j,k \lambda_j\} - \gamma_j) = k+(\min\{k \nu_j,k \lambda_j\} - \gamma_j).
\end{multline}
Therefore $\min\{k\nu_j,k\lambda_j\}-\gamma_j=0$. That is,
\[[(k\nu\wedge k\lambda)/ \gamma]_j=0.\]  
As a result, $[(k\nu\wedge k\lambda)/ \gamma]_i> 0$ only if $\nu_i = \lambda_i>0$. Hence by (\ref{eqn:May12xyy}) there are at least $\frac{t - |\nu \Delta \lambda|}{2}$ many $i$
with $\nu_i=\lambda_i>0$.
\qed

By Claim~\ref{May4rrr}, we may  define $\overline{\gamma}$ to be $\nu \wedge \lambda$ with one box removed from the southmost $\frac{t - |\nu \Delta \lambda|}{2}$ rows $i$ such that $\nu_i = \lambda_i>0$. 
It follows from (\ref{May4sss}) that $\nu/{\overline\gamma}$ and $\lambda/{\overline\gamma}$ are vertical strips.
Now, since $|\nu|+|\lambda|=2|\nu\wedge\lambda|+|\nu\Delta\lambda|$,
\begin{align*}
|\nu/\overline{\gamma}|& =  |\nu|-|\nu\wedge\lambda|+\frac{t-|\nu\Delta\lambda|}{2}\\
\ & = \frac{2|\nu|-2|\nu\wedge\lambda|+t-|\nu\Delta\lambda|}{2}\\
\ & = \frac{|\nu|-|\lambda|+t+(|\nu|+|\lambda|-2|\nu\wedge\lambda|-|\nu\Delta\lambda|)}{2}\\
\ & = \frac{|\nu|-|\lambda|+t}{2}.
\end{align*}
Similarly, $|\lambda/\overline{\gamma}|=\frac{|\lambda|-|\nu|+t}{2}$. Therefore, the
(column version) of the classical Pieri rule (\ref{eqn:usualPieri}) shows that
\[((1^{(t+|\nu|-|\lambda|)/2}), (1^{(t+|\lambda|-|\nu|)/2}),\overline{\gamma})\] 
is a witness for $N_{(1^t),\nu,\lambda}>0$.

The proof where one of $\mu,\nu,\lambda$ is a single row is similar to the above argument, except simpler.
Therefore we only sketch the necessary changes and leave the details to the reader. By Proposition~\ref{thm:Pieri},
we have $|\nu_i'-\lambda_i'|\leq 1$;
this is the analogue of (\ref{May4sss}). By the same reasoning, $\frac{t-|\nu\Delta\lambda|}{2}\in {\mathbb Z}_{\geq 0}$. The column version of Claim~\ref{May4rrr} states that there are at least
$\frac{t-|\nu\Delta\lambda|}{2}$ indices $i$ such that $\nu'_i=\lambda'_i>0$; it is proved using a different
Pigeonhole argument.  Given this claim, one
defines ${\widehat \gamma}$ be removing a single box from the eastmost $\frac{t-|\nu\Delta\lambda|}{2}$
columns such that $\nu'_i=\lambda_i'$. Then one concludes in the same way.\end{proof}

\subsection{Horn and (extended) Weyl inequalities}
 
Let $[n]:=\{1,2,\ldots n\}$.  For any
\[I=\{i_1<i_2< \cdots<i_d\}\subseteq [n]\]
define the partition
\[
 \tau(I):=(i_d-d\geq \cdots\geq i_2-2 \geq i_1-1).
\]
This bijects subsets of $[n]$ of cardinality $d$ with partitions whose Young
diagrams are contained in a $d\times(n-d)$ rectangle. The following combines the main
results of A.~Klyachko \cite{Klyachko} and A.~Knutson-T.~Tao \cite{KT}.

\begin{theorem}\label{thm:classicalHorn}(\cite{Klyachko}, \cite{KT})
Let $\lambda,\mu,\nu \in {\sf Par}_n$ such that
$|\lambda|+|\mu|=|\nu|$. Then
$c_{\mu,\nu}^{\lambda}>0$ if and only if
for every $d<n$, and every triple of subsets $I,J,K\subseteq [n]$ of cardinality $d$ such that $c_{\tau(I),
\tau(J)}^{\tau(K)}> 0$,
\begin{equation}\label{eq:ineq}
\sum_{k\in K}\lambda_k\leq
\sum_{i\in
I}\mu_i+\sum_{j\in J}\nu_j.\end{equation}
\end{theorem}

The inequalities (\ref{eq:ineq}) are  the \emph{Horn inequalities} \cite{Horn}. 

\begin{proposition}
\label{prop:Horn}
Let $\mu,\nu,\lambda\in {\sf Par}_n$ such that $N_{\mu,\nu,\lambda}>0$.  Then 
the Horn inequalities (\ref{eq:ineq}) hold.
\end{proposition}
\begin{proof}
Since $N_{\mu,\nu,\lambda}> 0$, there exists $\alpha, \beta, \gamma$ such that $c_{\alpha,\beta}^\mu,c_{\alpha,\gamma}^\nu, c_{\beta, \gamma}^\lambda>0$.  

By Theorem~\ref{thm:classicalHorn},
 $(\mu, \alpha, \beta)$ satisfies the Horn inequalities (\ref{eq:ineq}).  Consider an arbitrary Horn inequality 
 associated to a triple of subsets $(I,J,K)$ as in Theorem~\ref{thm:classicalHorn}.
\[\sum_{k\in K} \lambda_k \leq \sum_{i\in I} \beta_i + \sum_{j\in J} \gamma_j.\]  Since $c_{\alpha,\gamma}^\nu >0$, $\gamma\subseteq \nu$ and so in particular $\gamma_j\leq \nu_j$ for all $j$, and similarly $\beta_i\leq \mu_i$, so 
\[\sum_{k\in K} \lambda_k \leq \sum_{i\in I} \mu_i + \sum_{j\in J} \nu_j.\] 
Hence $(\mu,\nu,\lambda)$ satisfies (\ref{eq:ineq}), as desired.
\end{proof}

Among the Horn inequalities are the  \emph{Weyl's inequalities} \cite{Weyl1912}. The latter inequalities state that a necessary
condition for $c_{\mu,\nu}^{\lambda}>0$ is
\begin{equation}
\label{eqn:Weyl}
\lambda_{i+j-1}\leq \nu_i+\mu_j \text{\ for $i+j-1\leq n$};
\end{equation}
we refer to \cite{Bhatia} and the references therein for an expository account.
When $n=2$, the Horn inequalities (\ref{eq:ineq}) and Weyl inequalities (\ref{eqn:Weyl}) coincide:
\begin{equation}
\label{eqn:ris2}
\lambda_1\leq \mu_1+\nu_1, \ \lambda_2\leq \mu_1+\nu_2, \ \lambda_2\leq \mu_2+\nu_1.
\end{equation}
Theorem~\ref{thm:classicalHorn} has been extended in a number of ways. For a
recent example, see work of N.~Ressayre \cite{Ress}, who gave inequalities
valid whenever the \emph{Kronecker coefficient} $g_{\mu,\nu,\lambda}>0$.

\begin{theorem}[Extended Weyl inequalities]
\label{thm:Horn}
Let $\mu,\nu,\lambda\in {\sf Par}_n$ and  
$1\leq k \leq i < j \leq l \leq n$, let $m = \min(i-k,l-j)$ and $M = \max(i-k,l-j)$.  If $N_{\mu,\nu,\lambda}>0$ then
\begin{equation}
\label{ijklInequality}
\mu_i - \mu_j \leq \lambda_l - \lambda_k + \nu_{m-p+1} + \nu_{M+p+2} \text{\ \ \ where $0\leq p\leq m$}.
\end{equation}
\end{theorem}
\begin{proof}
Since $N_{\mu,\nu,\lambda}> 0$, there exists $\alpha, \beta, \gamma$ such that $c_{\alpha,\beta}^\mu,c_{\alpha,\gamma}^\nu, c_{\beta, \gamma}^\lambda>0$. 
By Theorem~\ref{thm:classicalHorn},  $(\mu, \alpha, \beta)$, $(\nu, \alpha, \gamma)$, $(\lambda, \beta, \gamma)$ all satisfy the Horn inequalities.  Therefore, by Weyl's inequalities (\ref{eqn:Weyl}), we have that 
\begin{equation}
\label{eqn:April18bbb}
\mu_i \leq \alpha_{i-k+1} + \beta_k
\text{ \ \ and \ \ } \lambda_l \leq \beta_j + \gamma_{l+1-j}.
\end{equation}
Additionally,
\[c^{\tau([n]\setminus\{j\})}_{\tau([n]\setminus\{j\}),\tau([n-1])}=c^{(1^{n-j})}_{(1^{n-j}),(0)} =1,\] so by 
Theorem~\ref{thm:classicalHorn} applied to $c_{\alpha,\beta}^{\mu}>0$, 
\begin{equation}
\label{eqn:April18aaa}
\sum_{a\not=j} \mu_a \leq \sum_{b\not= n} \alpha_b + \sum_{c\not= j} \beta_c.
\end{equation} 
 Subtracting (\ref{eqn:April18aaa}) from 
\[\sum_a \mu_a = \sum_b \alpha_b + \sum_c \beta_c,\] 
gives 
\begin{equation}
\label{eqn:April18ccc}
\mu_j \geq \alpha_n + \beta_j.
\end{equation}
By the same logic, 
\begin{equation}
\label{eqn:April18ddd}
\lambda_k \geq \beta_k + \gamma_n.
\end{equation}
Also, by treating $\alpha$, $\gamma$, and $\nu$ as partitions of $n+1$ rows with $\alpha_{n+1} = \gamma_{n+1} = \nu_{n+1}=0$, we have that
\[c^{\tau([n+1]\setminus\{m-p+1,M+p+2\})}_{\tau([n+1]\setminus\{i-k+1,n+1\}),\tau([n+1]\setminus\{l-j+1,n+1\})}  = c^{(2^{n-1-M-p})\cup (1^{2p+M-m})}_{(1^{n-1-(i-k)}),(1^{n-1-(l-j)})}=1.\]
Thus, Theorem~\ref{thm:classicalHorn} applied to $c_{\alpha,\gamma}^{\nu}>0$ gives
\begin{equation}
\label{eqn:April20aaa}
\sum_{a\not\in \{m-p+1,M+p+2\}} \nu_a \leq \sum_{b\not\in \{i-k+1,n+1\}} \alpha_b + \sum_{c\not\in \{l-j+1,n+1\}} \gamma_c.
\end{equation}
Subtracting (\ref{eqn:April20aaa}) from 
\[\sum_a \nu_a = \sum_b \alpha_b + \sum_c \gamma_c\]
gives
\begin{equation}
\label{eqn:April20bbb}
\alpha_{i-k+1} + \gamma_{l-j+1} = \alpha_{i-k+1} + \alpha_{n+1} + \gamma_{l-j+1} + \gamma_{n+1} \leq \nu_{m-p+1} + \nu_{M+p+2}
\end{equation}
Therefore, combining (\ref{eqn:April18bbb}),
(\ref{eqn:April18ccc}) and (\ref{eqn:April18ddd}) gives the first inequality below:
\begin{align*}
\mu_i -\mu_j + \lambda_l-\lambda_k&\leq (\alpha_{i-k+1} + \beta_k) - (\alpha_n +\beta_j)+ (\beta_j + \gamma_{l+1-j}) - (\beta_k + \gamma_n) \\
& = \alpha_{i-k+1} - \alpha_n + \gamma_{l+1-j} - \gamma_n\\
& \leq \alpha_{i-k+1} + \gamma_{l+1-j} \\
& \leq \nu_{m-p+1} + \nu_{M+p+2},
\end{align*}
where we have just applied (\ref{eqn:April20bbb}). This completes
the derivation of (\ref{ijklInequality}).
\end{proof}

\begin{corollary}
\label{cor:cyclic}
The inequalities (\ref{eq:ineq}) and (\ref{ijklInequality}),
where the roles of $(\mu,\nu,\lambda)$ are interchanged under all 
${\mathfrak S}_3$-permutations, also hold whenever $N_{\mu,\nu,\lambda}>0$.
\end{corollary}
\begin{proof}
Combine Lemma~\ref{prop:easy}(I) with Proposition~\ref{prop:Horn} and Theorem~\ref{thm:Horn}.
\end{proof}

Just as the Weyl inequalities are necessary and sufficient to characterize ${\sf LR}_2$, we now show that
the (extended) Weyl inequalities (together with symmetries given by Corollary~\ref{cor:cyclic}) are necessary and sufficient to describe ${\sf NL}_2$. 
 
\begin{theorem}\label{prop:r=2,def=k}
Suppose $\lambda,\mu, \nu\in {\sf Par}_2$ satisfies
$|\lambda|+|\mu|+|\nu|\equiv 0 \ (\!\!\!\!\mod 2)$ and
 the triangle inequalities. Then
$(\mu,\nu,\lambda)\in {\sf NL}_2$ if and only if this list of linear inequalities holds:
\begin{equation}\label{eqn:the2horn1}
\lambda_1\leq \mu_1+\nu_1, \ \nu_1\leq \lambda_1+\mu_1, \ \mu_1\leq \lambda_1+\nu_1
\end{equation}
\begin{equation}\label{eqn:the2horn2}
\lambda_2\leq \mu_1+\nu_2, \ \nu_2\leq \lambda_1+\mu_2, \ \mu_2\leq \lambda_1+\nu_2
\end{equation}
\begin{equation}\label{eqn:the2horn3}
\lambda_2\leq \mu_2+\nu_1, \ \nu_2\leq \lambda_2+\mu_1, \ \mu_2\leq \lambda_2+\nu_1
\end{equation}
\begin{align}
\label{ineq:lin}
    \nu_1 \!-\! \nu_2 \leq \mu_1 \!+\! \mu_2 + \lambda_1 - \lambda_2, \ \mu_1-\mu_2\leq \lambda_1+\lambda_2+\nu_1-\nu_2, \ \lambda_1-\lambda_2\leq \nu_1+\nu_2+\mu_1-\mu_2\\
    \lambda_1-\lambda_2\leq \mu_1+\mu_2+\nu_1-\nu_2, \mu_1-\mu_2\leq \nu_1+\nu_2+\lambda_1-\lambda_2,
    \nu_1-\nu_2\leq \lambda_1+\lambda_2+\mu_1-\mu_2.\nonumber
\end{align}
\end{theorem}

Above, (\ref{eqn:the2horn1}), (\ref{eqn:the2horn2}), (\ref{eqn:the2horn3}) are the $n=2$ Horn/Weyl inequalities
(\ref{eqn:ris2}) and their symmetric analogues. (\ref{ineq:lin}) represents (up to symmetry) the unique inequality of the form
(\ref{ijklInequality}) for this case.

Theorem~\ref{prop:r=2,def=k} implies another case of Conjectures~\ref{theconj} and~\ref{theconj'}:
\begin{corollary}\label{n=2sat}
Conjectures~\ref{theconj} and~\ref{theconj'} hold when $n=2$.
\end{corollary}
\begin{proof}
Suppose that $|\lambda|+|\mu|+|\nu|\equiv 0 \ (\!\!\!\!\mod 2)$ and $N_{k\mu,k\nu,k\lambda}>0$. By
Theorem~\ref{prop:r=2,def=k}, $(k\mu,k\nu,k\lambda)$ satisfies 
(\ref{eqn:the2horn1}), (\ref{eqn:the2horn2}), (\ref{eqn:the2horn3}) and (\ref{ineq:lin}) after the substitution
\[\mu\mapsto k\mu, \nu\mapsto k\nu, \lambda\mapsto k\lambda.\]  These
inequalities are homogeneous in $\lambda_i,\mu_i,\nu_i$. Hence $(\mu,\nu,\lambda)$ satisfies
(\ref{eqn:the2horn1}), (\ref{eqn:the2horn2}), (\ref{eqn:the2horn3}) and (\ref{ineq:lin}). Therefore by
the ``$\Leftarrow$'' direction of Theorem~\ref{prop:r=2,def=k}, $N_{\mu,\nu,\lambda}>0$, as required.
\end{proof}

The classical Weyl inequalities do not characterize ${\sf LR}_3$. Analogously, the extended Weyl inequalities
(combined with Proposition~\ref{prop:Horn} and Corollary~\ref{cor:cyclic}) are not sufficient to characterize
${\sf NL}_3$. An example is $\mu=(6,0,0),\nu=(4,2,2)$ and $\lambda=(4,4,0)$. However,
we have an additional list of inequalities that should close the gap in this case. 
We plan to address this issue (and more) in a sequel. For now, we restrict to proving Theorem~\ref{prop:r=2,def=k}, to illustrate a general strategy.

\noindent \emph{Proof of Theorem~\ref{prop:r=2,def=k}:} The ``$\Rightarrow$'' direction is by Proposition~\ref{prop:Horn},
 Theorem~\ref{thm:Horn}, and Corollary~\ref{cor:cyclic}. To prove the converse, let $(\lambda,\mu,\nu)\in {\sf Par}_2$
be such that $|\lambda|+|\mu|+|\nu|\equiv 0 \ (\!\!\!\!\mod 2)$ and $N_{\mu,\nu,\lambda}=0$. We
now show that either one of the triangle inequalities, or an inequality from 
(\ref{eqn:the2horn1})-(\ref{ineq:lin}), is violated. 

\begin{claim}\label{lemma:<symdiff}
If $|\lambda| < |\mu \Delta \nu|$, either a triangle inequality or an inequality from (\ref{ineq:lin}) is violated.
\end{claim}
\noindent
\emph{Proof of Claim~\ref{lemma:<symdiff}:}
By Lemma~\ref{prop:easy}(I), we may assume without loss that $\nu_1 \geq \mu_1$. If $\nu_2 \geq \mu_2$, then $|\mu \Delta \nu| = |\nu|-|\mu|$. Combining this with the hypothesis $|\lambda| < |\mu \Delta \nu|$ we obtain a failure of the triangle inequality $|\lambda|+|\mu| \geq |\nu|$. If $\nu_2 < \mu_2$, then 
\[|\mu \Delta \nu| = \nu_1-\mu_1+\mu_2-\nu_2.\] 
Now, $|\lambda| < |\mu \Delta \nu|$ implies that 
\[\nu_1-\nu_2 > \lambda_1+\lambda_2+\mu_1-\mu_2\]
which violates the sixth equation of (\ref{ineq:lin}).\qed

By Claim~\ref{lemma:<symdiff}, we may henceforth assume that 
\begin{equation}
\label{eqn:theassump}
|\mu \Delta \nu| \leq |\lambda| \leq |\mu|+|\nu|.
\end{equation}
Let 
\begin{equation}
\label{eqn:May1aaab}
k = \frac{|\mu|+|\nu|-|\lambda|}{2} \geq 0;
\end{equation} 
$k\in {\mathbb Z}$ by the hypothesis that $|\lambda|+|\mu|+|\nu|\equiv 0 \ (\!\!\!\!\mod 2)$. For future use, we record this rewriting of (\ref{eqn:May1aaab}): 
\begin{equation}
\label{eqn:May1aaa}
\lambda_1+\lambda_2=\mu_1+\mu_2+\nu_1+\nu_2-2k.
\end{equation}

A pair $(\mu^{\downarrow k},\nu^{\downarrow k})\in {\sf Par}_2$ is \emph{valid} if there exists $\alpha\in {\sf Par}_2$ with
$|\alpha|=k$ such that  $c_{\alpha,\mu^{\downarrow k}}^\mu>0$ and $c_{\alpha,\nu^{\downarrow k}}^\nu>0$ (equivalently, $\mu^{\downarrow k} \subset \mu, \ \nu^{\downarrow k} \subset \nu$ with $ |\mu/\mu^{\downarrow k}| = |\nu/\nu^{\downarrow k}| = k$, and the two skew shapes $\mu/\mu^{\downarrow k} \text{ and }\nu/\nu^{\downarrow k}$  each have a LR tableau of the same content $\alpha$).

\begin{claim}\label{claim:April30aaa}
 A valid pair $(\mu^{\downarrow k},\nu^{\downarrow k})$ exists. Moreover, 
 \begin{equation}
 \label{eqn:kmin}
 k\leq |\mu\wedge \nu|=\min(\mu_1,\nu_1)+\min(\mu_2,\nu_2).
 \end{equation}
  \end{claim}
 \noindent
 \emph{Proof of Claim~\ref{claim:April30aaa}:} By (\ref{eqn:theassump}), $|\lambda| \geq |\mu \Delta \nu|$. Thus
 existence follows from Theorem~\ref{thm:interval}(I) combined with (\ref{eqn:Newell-Littlewood}). 
 (\ref{eqn:kmin}) holds since 
$|\mu\wedge\nu|=\frac{|\mu|+|\nu|-|\mu\Delta\nu|}{2}\geq \frac{|\mu|+|\nu|-|\lambda|}{2}:=k$.
\qed

For $i=1,2$, let $k_i$ and $l_i$ to be the number of boxes in row $i$ of the skew shapes $\mu/\mu^{\downarrow k}$ and $\nu/\nu^{\downarrow k}$ respectively.

\begin{claim}
\label{claim:May2nnn}
If ($\mu^{\downarrow k},\nu^{\downarrow k}$) is valid then at least one of the following
inequalities holds:
\begin{equation}\label{ineq1:def=k}
    \lambda_1 > \mu_1 +\nu_1 - k_1-l_1
\end{equation}
\begin{equation}\label{ineq2:def=k}
    \lambda_2 > \mu_1 +\nu_2 - k_1-l_2
\end{equation}
\begin{equation}\label{ineq3:def=k}
    \lambda_2 > \mu_2 +\nu_1 - k_2-l_1.
\end{equation}
\end{claim}
\noindent\emph{Proof of Claim~\ref{claim:May2nnn}:}
By (\ref{eqn:Newell-Littlewood}), $N_{\mu,\nu,\lambda} = 0 \iff c^{\lambda}_{\mu^{\downarrow k}\nu^{\downarrow k}}=0$ whenever ($\mu^{\downarrow k}\nu^{\downarrow k}$) is a valid pair.  
Now the claim holds by the $n=2$ case of
 Theorem~\ref{thm:classicalHorn} (see (\ref{eqn:ris2})).\qed

\begin{claim}\label{claim:2rowLR}
Suppose $\mu^{\downarrow k}=(\mu_1-k_1,\mu_2-k_2), \nu^{\downarrow k}=(\nu_1-l_1,\nu_2-l_2)$  
and $\alpha=(\alpha_1,\alpha_2)\in {\mathbb Z}^2$. Then ($\mu^{\downarrow k}, \nu^{\downarrow k}$) is a valid pair of content $\alpha$ if and only if 
\begin{itemize}
\item[(I)] $\mu^{\downarrow k}, \nu^{\downarrow k} \in {\sf Par}_2$;
\item[(II)] $\alpha\in {\sf Par}_2$;
\item[(III)] $k_1,k_2,l_1,l_2\in {\mathbb Z}_{\geq 0}$;
\item[(IV)] $k_1+k_2=l_1+l_2=\alpha_1+\alpha_2=k$; 
\item[(V)] $k_1,k_2\geq \alpha_2$ and $l_1,l_2\geq \alpha_2$; and
\item[(VI)] $\alpha_2 +(\mu_1-\mu_2) \geq k_1$ and $\alpha_2 +(\nu_1-\nu_2) \geq l_1$.
\end{itemize}
\end{claim}
\noindent\emph{Proof of Claim~\ref{claim:2rowLR}:}
($\Leftarrow$) We construct a LR tableaux $T$ of shape $\mu/\mu^{\downarrow k}$ of content $\alpha$. Conditions (I), (III) guarantees
this is a skew-shape. Fill the $k_1$ boxes
of the first row of $\mu/\mu^{\downarrow k}$ with $1$'s. Since by (V), $k_2\geq \alpha_2$, we can fill
the rightmost $\alpha_2$ boxes of the second row of $\mu/\mu^{\downarrow k}$ with $2$'s. Then fill
the remaining boxes of that row with $1$'s.  $T$ is clearly row semistandard. It is column semistandard because of (VI). It is ballot by (II) and the condition $k_1\geq \alpha_2$ of (V). Finally the content of $T$ is $\alpha$ by (IV).
Thus $c_{\mu^{\downarrow k},\alpha}^{\mu}>0$. Similarly, we show $c_{\nu^{\downarrow k},\alpha}^{\nu}>0$.

($\Rightarrow$) If ($\mu^{\downarrow k}, \nu^{\downarrow k}$) is a valid pair of content $\alpha$ then there exists
LR tableaux $T,U$ of shapes $\mu/\mu^{\downarrow k}$ and $\nu/\nu^{\downarrow k}$ (respectively), and of
common content $\alpha$. Now the conditions follow by reversing the reasoning in the above paragraph.
\qed

\begin{claim}\label{claim:ineq1}
If (\ref{ineq1:def=k}) holds for every valid pair ($\mu^{\downarrow k}$, $\nu^{\downarrow k}$) 
then an inequality from (\ref{eqn:the2horn1})-(\ref{ineq:lin}) is violated.
\end{claim}
\noindent
\emph{Proof of Claim~\ref{claim:ineq1}:} By Lemma~\ref{prop:easy}(I), we may assume, without loss, that $\mu_2 \geq \nu_2$. In each case below, it is straightforward to verify the conditions
(I)-(VI) of Claim~\ref{claim:2rowLR}, so this is left mostly to the reader.

\noindent \emph{Case 1 ($\min(\mu_2,\nu_1,k)=\nu_1$):} Consider $\mu^{\downarrow k} = (\mu_1-(k-\nu_1),\mu_2-\nu_1)$ and $\nu^{\downarrow k} = (\nu_1-(k-\nu_2))$. We point out that, here and elsewhere, (\ref{eqn:kmin}) is relevant to checking Claim~\ref{claim:2rowLR}; in this case condition (I). Specifically, 
$\mu^{\downarrow k}_1, \nu^{\downarrow k}_1\geq 0$ by (\ref{eqn:kmin}). In addition $\mu^{\downarrow k}_1\geq
\mu^{\downarrow k}_2$ since 
\[\mu_1-(k-\nu_1)-(\mu_2-\nu_1) \geq \mu_1-\mu_2 + |\nu| - k \geq 0\]
(again by (\ref{eqn:kmin})). It follows that ($\mu^{\downarrow k}, \nu^{\downarrow k}$) is a valid pair of content $\alpha=(\nu_1,k-\nu_1)$.
 In this case we have $k_2 = \nu_1$ and $l_2 = \nu_2$ and thus $k_2 + l_2 = |\nu|$. Now 
 by (\ref{eqn:May1aaa}), (\ref{ineq1:def=k}), and Claim~\ref{claim:2rowLR}(IV),  
 \begin{equation}
 \label{eqn:May5bbb}
 \lambda_2<\mu_2 +\nu_2 -k_2-l_2.
 \end{equation}
 Hence,
  $\lambda_2 + \nu_1 < (\mu_2 +\nu_2 -k_2-l_2) +\nu_1=
  \mu_2 + \nu_2 -|\nu| + \nu_1 = \mu_2$. This violates the third inequality of (\ref{eqn:the2horn3}).

\noindent \emph{Case 2a ($\min(\mu_2,\nu_1,k)=k$ and $\nu_2\geq k$):}  $\mu^{\downarrow k}=(\mu_1,\mu_2-k)$ and $\nu^{\downarrow k} = (\nu_1,\nu_2-k)$ is a valid pair of content $\alpha = (k)$. Here $k_1=l_1=0$. Hence
(\ref{ineq1:def=k}) states $\lambda_1 > \mu_1+\nu_1$, violating (\ref{eqn:the2horn1}).

\noindent \emph{Case 2b ($\min(\mu_2,\nu_1,k)=k$ and $k\geq \nu_2$):} $\mu^{\downarrow k} = (\mu_1,\mu_2-k)$ and $\nu^{\downarrow k} = (\nu_1-(k-\nu_2))$ is a valid pair with $\alpha = (k)$. Here $k_1 = 0, k_2 = k, l_1 = k-\nu_2$ and $ l_2 = \nu_2$. By (\ref{ineq1:def=k}) and  (\ref{eqn:May5bbb}),
\begin{align*}
\lambda_1-\lambda_2 & > (\mu_1+\nu_1-k_1-l_1)-(\mu_2+\nu_2-k_2-l_2) \\
&= \mu_1 +\nu_1-(k-\nu_2) -\mu_2-\nu_2 + k + \nu_2\\
&=\nu_1 +\nu_2 + \mu_1-\mu_2
\end{align*}
which violates the third inequality from (\ref{ineq:lin}).

\noindent \emph{Case 3 ($\min(\mu_2,\nu_1,k)=\mu_2$):} Let $\mu^{\downarrow k} = (\mu_1-(k-\mu_2))$, $\nu^{\downarrow k} = (\nu_1 - (k-\nu_2))$. By (\ref{eqn:kmin}), $\mu_2 \geq \nu_2 \geq k-\min \{\mu_1,\nu_1\}$. Using this, one checks $(\mu^{\downarrow k},\nu^{\downarrow k})$ is valid
of content $\alpha = (\min \{\mu_1,\nu_1\}, k-\min\{\mu_1,\nu_1\})$.
Here, $k_2 = \mu_2$ and $l_2 = \nu_2$.
Hence by (\ref{eqn:May5bbb}), $\lambda_2 < \mu_2 + \nu_2 -k_2-l_2 = 0$ contradicts that $\lambda\in {\sf Par}_2$.
\qed

Introduce the quantity
\[\Delta(\mu^{\downarrow k}, \nu^{\downarrow k}):=(\mu_1+\nu_2-k_1-l_2)-(\mu_2+\nu_1-k_2-l_1).\]

\begin{claim}\label{obser:2&3}
Suppose ($\mu^{\downarrow k}, \nu^{\downarrow k}$) is a valid pair
such that $|\Delta(\mu^{\downarrow k}, \nu^{\downarrow k})| \leq 1$. Then 
(\ref{ineq2:def=k}) and (\ref{ineq3:def=k}) are violated.
\end{claim}
\noindent
\emph{Proof of Claim~\ref{obser:2&3}:}  If (\ref{ineq2:def=k}) holds, by (\ref{eqn:May1aaa}) and
Claim~\ref{claim:2rowLR}(IV) we obtain
\[\lambda_1 \leq \mu_2+\nu_1-k_2-l_1-1 \leq \mu_1+\nu_2-k_1-l_2 < \lambda_2,\] 
which is a contradiction of
$\lambda\in {\sf Par}_2$. Similarly, if (\ref{ineq3:def=k}) holds then  
\[\lambda_1 \leq \mu_1+\nu_2-k_1-l_2-1 \leq \mu_2+\nu_1-k_2-l_1 < \lambda_2,\] 
giving the same contradiction.\qed

\begin{claim}\label{claim:ineq2&3}
Suppose $(\mu^{\downarrow k},\nu^{\downarrow k})$ and $({\widetilde \mu}^{\downarrow k}, {\widetilde \nu}^{\downarrow k})$
are valid pairs of content $\alpha$ and $\widetilde \alpha$, respectively. There is a sequence of valid pairs 
\[(\mu^{\downarrow k}_{(0)},\nu^{\downarrow k}_{(0)})=(\mu^{\downarrow k},\nu^{\downarrow k}), (\mu^{\downarrow k}_{(1)},\nu^{\downarrow k}_{(1)}), \ldots (\mu^{\downarrow k}_{(m)},\nu^{\downarrow k}_{(m)}) = ({\widetilde \mu}^{\downarrow k},{\widetilde \nu}^{\downarrow k})\] 
of contents $\alpha^{(0)}=\alpha, \alpha^{(1)},\ldots, \alpha^{(m)}={\widetilde \alpha}$ (respectively) such that for all $i \in [m]$,
\begin{equation}
\label{eqn:May1hhh}
|\Delta(\mu^{\downarrow k}_{(i)},\nu^{\downarrow k}_{(i)})-\Delta(\mu^{\downarrow k}_{(i-1)},\nu^{\downarrow k}_{(i-1)})|\leq 2.
\end{equation}
\end{claim}

\noindent \emph{Proof of Claim~\ref{claim:ineq2&3}:} First suppose that $\alpha = {\widetilde \alpha}$. 
By exchanging the roles of ($\mu^{\downarrow k},\nu^{\downarrow k}$) and (${\widetilde \mu}^{\downarrow k}, {\widetilde \nu}^{\downarrow k}$) if necessary, we may assume that $k_1 - {\widetilde k}_1 = j \geq 0$. Define 
\[\mu^{\downarrow k}_{(i+1)}=(\mu^{\downarrow k}_{(i)1}+1, \mu^{\downarrow k}_{(i)2}-1)\]
$0\leq i< j$. Also, set $\nu^{\downarrow k}_{(i)} = \nu^{\downarrow k}_{(0)}$ for all $0 < i \leq j$.
By definition of $j$, $\mu^{\downarrow k}_{(j)} = {\widetilde \mu}^{\downarrow k}$. Moving a single box
at a time, we construct $\nu^{\downarrow k}_{(i)}$ similarly for $i > j$ such that when $i=m$ we obtain
$\nu^{\downarrow k}$ (and we set $\mu_{i}^{\downarrow k}={\widetilde \mu}^{\downarrow k}$ for $j<i\leq m$). 
More precisely if $l_1={\widetilde l}_1$ then $j=m$. If $l_1>{\widetilde l}_1$ then 
set $\nu_{(i+1)}^{\downarrow k}=(\nu_{(i)1}^{\downarrow k}+1,\nu_{(i)2}^{\downarrow k}-1)$ for $j\leq i<m$. 
Finally if $l_1<{\widetilde l}_1$ we set $\nu_{(i+1)}^{\downarrow k}=(\nu_{(i)1}^{\downarrow k}-1,\nu_{(i)2}^{\downarrow k}+1)$ for $j\leq i<m$. 

Set $\alpha^{(i)}=\alpha={\widetilde \alpha}$ for $0\leq i\leq m$. It is a straightforward induction argument to
see that each $(\mu_{(i)}^{\downarrow k},\nu_{(i)}^{\downarrow k})$ is valid of content $\alpha^{(i)}$. Finally,
by construction, 
\begin{equation}
\label{eqn:April30xyz}
|(k^{(i)}_2-k^{(i)}_1+l^{(i)}_1-l^{(i)}_2)-(k^{(i-1)}_2-k^{(i-1)}_1+l^{(i-1)}_1-l^{(i-1)}_2)| =2,
\end{equation}
which implies (\ref{eqn:May1hhh}).

Now suppose that $\alpha \neq {\widetilde \alpha}$. We assume without loss of generality that $\alpha_2 > {\widetilde \alpha}_2$.  Let $m^{\star}:=\alpha_2-{\widetilde \alpha_2}>0$. Then, for $0\leq i \leq m^{\star}-1$ set
\begin{equation}
\alpha^{(i+1)} = (\alpha^{(i)}_{1}+1,\alpha^{(i)}_{2}-1),
\end{equation}
\begin{equation}
\label{eqn:May8sss}
\mu^{\downarrow k}_{(i+1)}=
\begin{cases}
\mu^{\downarrow k}_{(i)} \text{ \ \ \ \ \ \ \ \ \ \ \ \ \ \ \ \ \ \ \ \ \ \ \ \ \ \ \ \ \  if $c^{\mu}_{\mu^{\downarrow k}_{(i)}, \alpha^{(i+1)}}>0$}

\\
(\mu^{\downarrow k}_{(i)1}+1, \mu^{\downarrow k}_{(i)2}-1) \text{\ \ \  otherwise,}
\end{cases}
\end{equation}
and
\begin{equation}
\label{eqn:May8ttt}
\nu^{\downarrow k}_{(i+1)}=
\begin{cases}
\nu^{\downarrow k}_{(i)} \text{ \ \ \ \ \ \ \ \ \ \ \ \ \ \ \ \ \ \ \ \ \ \ \ \ \ \ \ \ \  if $c^{\nu}_{\nu^{\downarrow k}_{(i)}, \alpha^{(i+1)}}>0$}

\\
(\nu^{\downarrow k}_{(i)1}+1, \nu^{\downarrow k}_{(i)2}-1) \text{\ \ \  otherwise.}
\end{cases}
\end{equation}

It is straightforward to check 
\[
|(k^{(i)}_2-k^{(i)}_1+l^{(i)}_1-l^{(i)}_2)-(k^{(i-1)}_2-k^{(i-1)}_1+l^{(i-1)}_1-l^{(i-1)}_2)| \in \{0,2\}\]
and hence (\ref{eqn:May1hhh}) holds.

Thus, it remains to show that ($\mu^{\downarrow k}_{(i+1)}, \nu^{\downarrow k}_{(i+1)}$) is a valid pair of content
$\alpha^{(i+1)}$. By definition, the only concern is if $\mu^{\downarrow k}_{(i+1)}$ (respectively, 
$\nu^{\downarrow k}_{(i+1)}$) is obtained by applying the second case of (\ref{eqn:May8sss})  
(respectively, (\ref{eqn:May8ttt})). Now, suppose we applied the second case of (\ref{eqn:May8sss})
to obtain $\mu^{\downarrow k}_{(i+1)}$. Since, by induction, ($\mu^{\downarrow k}_{(i)}, \nu^{\downarrow k}_{(i)}$) is valid of content $\alpha^{(i)}$, there exists an LR tableau $T$ of shape 
$\mu/\mu^{\downarrow k}_{(i)}$ of content $\alpha^{(i)}$. The assumption $\alpha_2>\widetilde{\alpha}_2$
implies $\alpha_1<\widetilde{\alpha}_1$. This combined  with the induction hypothesis, 
the fact that $\mu^{\downarrow k}_{(i),2} + \alpha^{(i)}_1 = \mu_1$ holds when
$c^{\mu}_{\mu^{\downarrow k}_{(i)}, \alpha^{(i+1)}}=0$, and $\mu_1 \geq  \widetilde{\alpha}_1 > \alpha^{(i)}_1$, shows
\[(\mu^{\downarrow k}_{(i)1}+1, \mu^{\downarrow k}_{(i)2}-1)\in {\sf Par}_2.\]
Now, define $T'$ by modifying 
$T$ as follows: Move the leftmost $1$ in the first row and place it to the left of the leftmost entry of the second row.
Then change the leftmost $2$ in the second row into a $1$.  

By definition of $m^{\star}$, and the existence of $T$, there exists a (leftmost) $1$ in the first row and a $2$ in the second row. Hence the modification is well-defined for $0\leq i<m$. Moreover, it is clear $T'$ is semistandard, of content $\alpha^{(i+1)}$ and has shape $\mu/\mu_{(i+1)}^{\downarrow k}$. That
$T'$ is ballot follows easily from the fact $T$ is ballot. Hence $T'$ is an LR tableau of the desired type.

In the same way, if $\nu_{(i+1)}^{\downarrow k}$ is obtained from $\nu_{(i)}^{\downarrow k}$ using the second case of (\ref{eqn:May8ttt}), we
can modify an LR tableau $U$ of shape $\nu/\nu^{\downarrow k}_{(i)}$ of content $\alpha^{(i)}$ into an LR
tableau of shape $\nu/\nu_{(i+1)}^{\downarrow k}$ and content $\alpha^{(i+1)}$.

Summarizing, irregardless of which cases of (\ref{eqn:May8sss}) and (\ref{eqn:May8ttt}) are used
at each stage, by induction, $(\mu_{(i+1)}^{\downarrow k},\nu_{(i+1)}^{\downarrow k})$ is valid of content 
$\alpha^{(i+1)}$. Moreover when $i+1=m^{\star}$, we arrive at 
$({\mu}^{\downarrow k}_{(m^{\star})}, { \nu}^{\downarrow k}_{(m^{\star})})$ of content
$\widetilde \alpha$. We have therefore reduced to the $\alpha={\widetilde \alpha}$ case above. Applying the
argument of that case we continue this 
sequence to $({\widetilde \mu}^{\downarrow k}, {\widetilde\nu}^{\downarrow k})$.
\qed

\begin{claim}\label{claim:ineqall2&3}
If all valid pairs ($\mu^{\downarrow k}, \nu^{\downarrow k}$) satisfy (\ref{ineq2:def=k}) or (\ref{ineq3:def=k}) then one of the inequalities from (\ref{eqn:the2horn1})-(\ref{ineq:lin}) is violated.
\end{claim}

\noindent \emph{Proof of Claim~\ref{claim:ineqall2&3}}: Claim~\ref{obser:2&3} says that $|\Delta(\mu^{\downarrow k},\nu^{\downarrow k})|\leq 1$ cannot occur.

If we have two valid pairs $(\mu^{\downarrow k}, \nu^{\downarrow k})$, $({\widetilde \mu}^{\downarrow k},{\widetilde \nu}^{\downarrow k})$ satisfying 
\[\Delta(\mu^{\downarrow k},\nu^{\downarrow k})<-1 \text{\ and $\Delta({\widetilde \mu}^{\downarrow k},{\widetilde \nu}^{\downarrow k})>1$,}\] 
then by Claim~\ref{claim:ineq2&3} there is a sequence $(\mu^{\downarrow k}_{(0)},\nu^{\downarrow k}_{(0)})
=(\mu^{\downarrow k},\nu^{\downarrow k})$, $(\mu^{\downarrow k}_{(1)},\nu^{\downarrow k}_{(1)})$ $\ldots$ 
$(\mu^{\downarrow k}_{(m)},\nu^{\downarrow k}_{(m)})= ({\widetilde \mu}^{\downarrow k},{\widetilde \nu}^{\downarrow k})$ such that $|\Delta(\mu^{\downarrow k}_{(i)},\nu^{\downarrow k}_{(i)})-\Delta(\mu^{\downarrow k}_{(i-1)},\nu^{\downarrow k}_{(i-1)})|\leq 2$ for all $i \in [m]$. Hence for some $j$,
$\Delta(\mu^{\downarrow k}_{(j)},\nu^{\downarrow k}_{(j)})\in \{-1,0,1\}$. However, in that case,
($\mu^{\downarrow k}_{(j)},\nu^{\downarrow k}_{(j)}$) contradicts our hypothesis, by
Claim~\ref{obser:2&3}. 

Since $\Delta(\mu^{\downarrow k}, \nu^{\downarrow k})=-\Delta(\nu^{\downarrow k}, \mu^{\downarrow k})$, by Lemma~\ref{prop:easy}(I), we may assume $\Delta(\mu^{\downarrow k},\nu^{\downarrow k})<-1$. By definition this means $\mu_1+\nu_2-k_1-l_2 < \mu_2+\nu_1-k_2-l_1$. If furthermore $\lambda_2 > \mu_2+\nu_1-k_2-l_1$ 
then $\lambda_2 > \mu_1+\nu_2-k_1-l_2$. That is, if ($\mu^{\downarrow k}, \nu^{\downarrow k}$) satisfies (\ref{ineq3:def=k}) then  ($\mu^{\downarrow k}, \nu^{\downarrow k}$) satisfies (\ref{ineq2:def=k}). 
Thus, henceforth we assume ($\mu^{\downarrow k}, \nu^{\downarrow k}$) satisfies 
\begin{equation}
\label{May14abc}
\Delta(\mu^{\downarrow k},\nu^{\downarrow k})<-1 \text{\ and (\ref{ineq2:def=k}).}
\end{equation}

We have four cases, depending on $k$. We appeal to Claim~\ref{claim:2rowLR} in each case.

\noindent \emph{Case 1 ($k \leq \mu_2,\nu_1-\nu_2$):} $\mu^{\downarrow k} = (\mu_1,\mu_2-k)$, $\nu^{\downarrow k} = (\nu_1-k,\nu_2)$ is a valid pair with content $\alpha = (k)$. 
We have $k_1 = l_2 = 0$ and hence (\ref{ineq2:def=k}) says $\lambda_2 > \mu_1+\nu_2$ violating (\ref{eqn:the2horn2}).

\noindent \emph{Case 2 ($\mu_2 < k \leq \mu_1,\nu_1-\nu_2$):} $\mu^{\downarrow k} = (\mu_1-(k-\mu_2))$, $\nu^{\downarrow k} = (\nu_1-k,\nu_2)$ is a valid pair with content $\alpha = (k)$. By (\ref{ineq2:def=k}) combined with (\ref{eqn:May1aaa}), 
\begin{equation}
\label{eqn:May7aaa}
\lambda_1<\mu_2+\nu_1-k_2-l_1.
\end{equation}
We will use this inequality here and in the cases below. In the present case,
$k_2=0, l_1=k$ and thus (\ref{eqn:May7aaa}) says $\lambda_2 > \mu_1+\nu_2-k+\mu_2$. Combining with
(\ref{ineq2:def=k}) gives
 \[\lambda_1-\lambda_2 < \nu_1-\nu_2-\mu_1-\mu_2,\]
which violates (\ref{ineq:lin}).

\noindent \emph{Case 3 ($\mu_1 < k \leq \nu_1-\nu_2$):} Since $\nu_2 \geq \alpha_2 \geq k-\mu_1$ and $\mu_1 \leq \nu_1-\nu_2+k-\mu_1$, we have a valid pair $\mu^{\downarrow k} = (\mu_1-(k-\mu_2)), \nu^{\downarrow k} = (\nu_1-\mu_1,\nu_2-(k-\mu_1))$ with content $\alpha = (\mu_1,k-\mu_1)$. We have $k_2 = \mu_2$ and $l_1 = \mu_1$ and thus by (\ref{eqn:May7aaa}), 
\[\lambda_1 < \mu_2+\nu_1-\mu_2-\mu_1 = \nu_1-\mu_1,\]
which violates (\ref{eqn:the2horn1}).\\
\noindent \emph{Case 4 ($k > \nu_1-\nu_2$):} Let 
\[\alpha = \left(\nu_1-\nu_2+\left\lceil \frac{k-\nu_1+\nu_2}{2}\right\rceil, \left\lfloor \frac{k-\nu_1+\nu_2}{2}\right\rfloor\right), \text{ and}\] 
\[\nu^{\downarrow k} = \left( \nu_2-\left\lfloor \frac{k-\nu_1+\nu_2}{2}\right\rfloor, \nu_2-\left\lceil \frac{k-\nu_1+\nu_2}{2}\right\rceil\right).\]

One can check that there is a LR tableau of shape $\nu/\nu^{\downarrow k}$ and content $\alpha$ by verifying the conditions (I)-(VI) of Claim~\ref{claim:2rowLR}. In particular $\alpha\subseteq \nu$. If $\alpha \subseteq \mu$ as well then since $s_{\mu/\alpha}\neq 0$, by (\ref{eqn:skew}) we can find $\mu^{\downarrow k}$ such that ($\mu^{\downarrow k},\nu^{\downarrow k}$) is valid of content $\alpha$. However, in that case 
\[(\nu_1-l_1)-(\nu_2-l_2) =\nu^{\downarrow k}_1 - \nu^{\downarrow k}_2\leq 1,\] 
and hence
\begin{align*}
\Delta(\mu^{\downarrow k},\nu^{\downarrow k}):= & \mu_1+\nu_2-k_1-l_2-(\mu_2+\nu_1-k_2-l_1)\\
 = &\mu_1-k_1-(\mu_2-k_2) +\nu_2-\nu_1+l_1-l_2\\
 = & (\mu^{\downarrow k}_1-\mu^{\downarrow k}_2) -[(\nu_1-l_1)-(\nu_2-l_2)]\\
 \geq &-1.
 \end{align*}
 This would contradict the assumption $\Delta(\mu^{\downarrow k},\nu^{\downarrow k})<-1$.
  Therefore we may assume either $\mu_1<\alpha_1$ or $\mu_2 < \alpha_2$. 
 
 First suppose $\mu_1 < \alpha_1$. Using this assumption, and the definition of $\alpha_1$ one verifies
 the conditions (II) and (VI) Claim~\ref{claim:2rowLR}. It follows that
\[\mu^{\downarrow k}_{(1)} = (\mu_1-k+\mu_2),\nu^{\downarrow k}_{(1)} = (\nu_1-\mu_1,\nu_2-(k-\mu_1))\]
is a valid pair with content $\overline{\alpha} = (\mu_1,k-\mu_1)$. 
Now we have $k_2 = \mu_2$ and $l_1 = \mu_1$ and thus (\ref{eqn:May7aaa}) states 
\[\lambda_1 < \mu_2+\nu_1-\mu_2-\mu_1 = \nu_1-\mu_1.\]
This violates the second inequality of (\ref{eqn:the2horn1}). 

Now suppose $\mu_2 < \alpha_2$. Using this assumption,
\[\mu^{\downarrow k}_{(2)} = (\mu_1-k+\mu_2),\nu^{\downarrow k}_{(2)} = 
(\nu_1-[\nu_1-\nu_2+\mu_2], \nu_2-[\nu_2-\nu_1+k-\mu_2])\]
gives a valid pair of content ${\overline \alpha} = (k-\mu_2,\mu_2)$. 
Now we have $k_2 = \mu_2$ and $l_1 = \nu_1-\nu_2+\mu_2$ and so here (\ref{eqn:May7aaa}) is
\[\lambda_1 < \mu_2+\nu_1-(\mu_2)-(\nu_1-\nu_2+\mu_2) = \nu_2-\mu_2.\]
This gives a violation of the second equation of (\ref{eqn:the2horn2}).\qed

\noindent
\emph{Conclusion of the proof of Theorem~\ref{prop:r=2,def=k}:} If all valid pairs satisfy
(\ref{ineq2:def=k}) or (\ref{ineq3:def=k}), we are done by Claim~\ref{claim:ineqall2&3}.
Since by Claim~\ref{claim:May2nnn}, at least one of (\ref{ineq1:def=k}), (\ref{ineq2:def=k}) or (\ref{ineq3:def=k}) holds for valid pairs,
we may assume there is a valid pair $(\mu^{\downarrow k}, \nu^{\downarrow k})$ such that (\ref{ineq1:def=k}) holds. If in fact, \emph{all} valid pairs satisfy (\ref{ineq1:def=k}), we are done by Claim~\ref{claim:ineq1}. Hence we may also suppose there is a valid pair $({\widetilde \mu}^{\downarrow k}, {\widetilde \nu}^{\downarrow k})$ that does not satisfy (\ref{ineq1:def=k}).

Let us consider the sequence of valid pairs 
\[(\mu^{\downarrow k}_{(0)}, \nu^{\downarrow k}_{(0)}):=
(\mu^{\downarrow k}, \nu^{\downarrow k})
, (\mu^{\downarrow k}_{(1)}, \nu^{\downarrow k}_{(1)}), \ldots , 
(\mu^{\downarrow k}_{(m)}, \nu^{\downarrow k}_{(m)}):=({\widetilde \mu}^{\downarrow k}, {\widetilde \nu}^{\downarrow k})
\] 
where $(\mu^{\downarrow k}_{(i)}, \nu^{\downarrow k}_{(i)})\mapsto (\mu^{\downarrow k}_{(i+1)}, \nu^{\downarrow k}_{(i+1)})$ by Claim~\ref{claim:ineq2&3}'s construction.

Combining the fact that $(\mu^{\downarrow k}_{(0)}, \nu^{\downarrow k}_{(0)})=(\mu^{\downarrow k}, \nu^{\downarrow k})$ is a valid pair satisfying (\ref{ineq1:def=k}) with (\ref{eqn:May1aaa}) and Claim~\ref{claim:2rowLR}(IV),
\[\lambda_2< \mu^{\downarrow k}_{(0) 2} + \nu^{\downarrow k}_{(0) 2}-2k+k_1+l_1<\mu^{\downarrow k}_{(0) 2} + \nu^{\downarrow k}_{(0) 2}.\]
Hence
\begin{equation}
\label{eqn:April30bbb}
\lambda_2 < \mu^{\downarrow k}_{(0) 2} + \nu^{\downarrow k}_{(0) 2} \leq \min\{\mu^{\downarrow k}_{(0) 1} + \nu^{\downarrow k}_{(0) 2}, \mu^{\downarrow k}_{(0) 2} + \nu^{\downarrow k}_{(0) 1}\}.
\end{equation}

By examining Claim~\ref{claim:ineq2&3}'s construction (for both $\alpha = {\widetilde \alpha}$ and $\alpha \neq {\widetilde \alpha}$), it is straightforward to see that
\begin{equation}\label{May14xyz}
    |\min\{\mu^{\downarrow k}_{(i) 1} + \nu^{\downarrow k}_{(i) 2}, \mu^{\downarrow k}_{(i) 2} + \nu^{\downarrow k}_{(i) 1}\}-\min\{\mu^{\downarrow k}_{(i+1) 1} + \nu^{\downarrow k}_{(i+1) 2}, \mu^{\downarrow k}_{(i+1) 2} + \nu^{\downarrow k}_{(i+1) 1}\}|\leq 1.
\end{equation}

Inductively, if (\ref{ineq1:def=k}) holds for ($\mu^{\downarrow k}_{(i)}, \nu^{\downarrow k}_{(i)}$), then by the same
reasoning as for (\ref{eqn:April30bbb}),
\begin{align*}
    \lambda_2 &\leq \mu^{\downarrow k}_{(i) 2} + \nu^{\downarrow k}_{(i) 2} - 1\\
    & \leq \min\{\mu^{\downarrow k}_{(i) 1} + \nu^{\downarrow k}_{(i) 2}, \mu^{\downarrow k}_{(i) 2} + \nu^{\downarrow k}_{(i) 1}\} - 1.
\end{align*}
Combining with (\ref{May14xyz}), we get
    \[\lambda_2 \leq \min\{\mu^{\downarrow k}_{(i+1) 1} + \nu^{\downarrow k}_{(i+1) 2}, \mu^{\downarrow k}_{(i+1) 2} + \nu^{\downarrow k}_{(i+1) 1}\}.\]
This means ($\mu^{\downarrow k}_{(i+1)}, \nu^{\downarrow k}_{(i+1)}$)
violates (\ref{ineq2:def=k}) and (\ref{ineq3:def=k}); consequently, (\ref{ineq1:def=k}) holds for this valid pair. Therefore by induction, ($\mu^{\downarrow k}_{(m)}, \nu^{\downarrow k}_{(m)}$) satisfies (\ref{ineq1:def=k}), which contradicts the choice of ($\mu^{\downarrow k}_{(m)}, \nu^{\downarrow k}_{(m)}$).\qed

\subsection{Refinements?}
A conjecture of W.~Fulton (proved in \cite{KTW}) states that 
\[ c_{\mu,\nu}^{\lambda}=1\implies c_{k\mu,k\nu}^{k\lambda}=1, \  \ \forall k\geq 1.\] 

\begin{example}[Counterexample to analogue of W.~Fulton's conjecture]
One checks that
\[N_{(1,1),(1,1),(1,1)}=(c_{(1),(1)}^{(1)})^3=1 \text{\ but 
$N_{(2,2),(2,2),(2,2)}=(c_{(1,1),(1,1)}^{(1,1)})^3+(c_{(2),(2)}^{(2)})^3=2$.}\]
Hence, the analogue of Fulton's conjecture for $N_{\nu,\mu,\lambda}$ is false.
\qed
\end{example}

Define a function 
\[{\mathfrak c}_{\mu,\nu}^{\lambda}:{\mathbb Z}_{\geq 1}\to {\mathbb N} \text{ \ 
by $k\mapsto c_{k\mu,k\nu}^{k\lambda}$}.\] 
A conjecture of R.~C.~King-C.~Tollu-F.~Toumazet \cite{King} asserts that this function is interpolated
by a polynomial with nonnegative rational coefficients. The polynomiality property was
proved by H.~Derksen-J.~Weyman \cite{Derksen.Weyman}. Consequently, 
${\mathfrak c}_{\mu,\nu}^{\lambda}$ is called the \emph{Littlewood-Richardson polynomial}.
(The positivity conjecture
remains open in general.)

Similarly, let us define the \emph{Newell-Littlewood function}:
\[{\mathfrak N}_{\mu,\nu,\lambda}:{\mathbb Z}_{\geq 1}\to {\mathbb N} \text{ by $k\mapsto N_{k\mu,k\nu,k\lambda}$.}\] 
The following shows that 
${\mathfrak N}_{\mu,\nu,\lambda}(k)$ cannot always be interpolated by a single
polynomial.

\begin{theorem}[Non-polynomiality]
\label{prop:floorex}
There exist $\lambda,\mu,\nu$ such that ${\mathfrak N}_{\mu, \nu,\lambda}(k)\not\in {\mathbb R}[k]$.
\end{theorem}
\begin{proof}
We will show
${\mathfrak N}_{(1,1),(1,1),(1,1)}(k)=\left\lceil{\frac{k+1}{2}}\right\rceil$,
which is clearly non-polynomial.

Let $\mu,\nu,\lambda=(1,1)$ and suppose $\alpha,\beta,\gamma$ satisfy
$c_{\alpha,\beta}^{k\mu}c_{\alpha,\gamma}^{k\nu}c_{\beta,\gamma}^{k\lambda}>0$, \emph{i.e.},
$c_{\alpha,\beta}^{(k,k)}c_{\alpha,\gamma}^{(k,k)}c_{\beta,\gamma}^{(k,k)}>0$. The claim is that the only possible $(\alpha,\beta,\gamma)$ are 
\begin{equation}
\label{eqn:thechoice}
\alpha=\beta=\gamma=(j,k-j) \text{\ where $\left\lfloor\frac{k+1}{2}\right\rfloor \leq j\leq k$,}
\end{equation}
and in this case the contribution to (\ref{eqn:Newell-Littlewood}) is 
$(c_{(j,k-j),(j,k-j)}^{(k,k)})^3=1$. 
This would complete the proof as there are $\lceil\frac{k+1}{2}\rceil$ such $j$. That 
$c_{(j,k-j),(j,k-j)}^{(k,k)}=1$ 
follows easily from the Littlewood-Richardson rule. Hence it only remains to rule out other possible $(\alpha,\beta,\gamma)$.
Indeed, given such a triple, since $c_{\alpha,\beta}^{(k,k)}>0$ we must have $|\alpha|+|\beta|=2k$. Similarly, we obtain  
$|\alpha|+|\gamma|=2k \text{ and $|\beta|+|\gamma|=2k$}$
which together imply 
$|\alpha|=|\beta|=|\gamma|=k$. 
To conclude, we apply another
fact about Littlewood-Richardson coefficients that has a Schubert calculus provenance.
That is, 
$c_{\alpha,\beta}^{(m-\ell)^{\ell}}=\delta_{\beta,\alpha^{\vee}}$ where
$\alpha^\vee$ is the $180$-degree rotation of $(m-\ell)^{\ell}\setminus \beta$ (as used in Claim~\ref{claim:Apr1abc}).\footnote{Let $\sigma_{\alpha}$
denote the Schubert class for $\alpha\subset (m-\ell)^\ell$. The underlying Schubert calculus
statement is that if $|\alpha|+|\beta|=\dim {\sf Gr}_{\ell}({\mathbb C}^m)(=\ell\times (m-\ell))$
then $\sigma_{\alpha}\cup\sigma_{\beta}=\delta_{\beta,\alpha^{\vee}}\sigma_{(m-\ell)^\ell}\in H^*({\sf Gr}_{\ell}({\mathbb C}^m))$.}
In our case $\ell=2$ and $m=k+2$; moreover $(j,k-j)^{\vee}=(j,k-j)$. From this, the result follows.
\end{proof}

\begin{example}
Let 
${\overline {\mathfrak N}_{\mu,\nu,\lambda}}(k):={\mathfrak N}_{\mu,\nu}^{\lambda}(2k-1),
{\widetilde {\mathfrak N}_{\mu,\nu,\lambda}}(k):={\mathfrak N}_{\mu,\nu,\lambda}(2k)$.
By Proposition~\ref{prop:floorex}, 
\[{\overline {\mathfrak N}_{(1,1),(1,1),(1,1)}}=k \text{\ and ${\widetilde {\mathfrak N}_{(1,1),(1,1),(1,1)}}=k+1$.}\] 
For another example, it seems
that 
\[{\overline {\mathfrak N}}_{(2,1,1),(2,1,1),(1,1,1,1)}=\frac{1}{3}k(k+2)(k+1) \text{\ and
${\widetilde {\mathfrak N}}_{(2,1,1),(2,1,1),(1,1,1,1)}=\frac{1}{6}(2k+3)(k+2)(k+1)$.}\]
This would suggest ${\overline {\mathfrak N}_{\mu,\nu,\lambda}}, {\widetilde {\mathfrak N}_{\mu,\nu,\lambda}}\in {\mathbb Q}_{\geq 0}[k]$. However, when $\lambda=\mu=\nu=(2,1,1)$, the values of ${\mathfrak N}_{\mu,\nu,\lambda}(k)$ for $k=1,2,\ldots,11$ are 
$4,18,51,141,315,676,1288,2370,4047,6720,10605$.
None of ${\mathfrak N}_{\mu,\nu,\lambda}, {\overline {\mathfrak N}_{\mu,\nu,\lambda}}, {\widetilde {\mathfrak N}_{\mu,\nu,\lambda}}$
seem to have a nice interpolation, although it is possible we do not have
sufficiently many values.
\qed
\end{example}

\subsection{Complexity of computing $N_{\mu,\nu,\lambda}$}
\label{sec:complexity}
Following  H.~Narayanan \cite{Narayanan}, T.~McAllister-J.~De Loera \cite{DeLoera}, and K.~D.~Mulmuley-H.~Narayanan-M.~Sohoni \cite{Mulmuley},
Theorem~\ref{thm:thepolytope} and Conjecture~\ref{theconj} have some implications
about the complexity of computing $N_{\mu,\nu,\lambda}$. For brevity, we limit ourselves to a sketch.

Given input $(\lambda,\mu,\nu)\in {\sf Par}_n$ (measured in terms of bit-size complexity)
there is the counting problem ${\tt NLvalue}$ which outputs $N_{\mu,\nu,\lambda}$.
By Lemma~\ref{prop:easy}(II), a subproblem is ${\tt LRvalue}$ (computation of 
$c_{\lambda,\mu}^{\nu}$). H.~Narayanan \cite{Narayanan} 
shows 
${\tt LRvalue}\in \#P\text{-complete}$ 
(thus, in particular, no polynomial time algorithm exists for this problem unless $P=NP$).
This implies ${\tt NLvalue}$ is $\#P$-hard. Theorem~\ref{thm:thepolytope} shows that
the problem is in $\#P$ since the vectors $(\alpha_i^j,\beta_i^j,\gamma_i^j)$ provide 
an efficient encoding of elements of a set counted by $N_{\mu,\nu,\lambda}$. Summarizing, 
\[{\tt NLvalue}\in \#P\text{-complete}.\]

The decision problem ${\tt NLnonzero}$ decides if
$N_{\mu,\nu,\lambda}>0$. Theorem~\ref{thm:thepolytope} implies 
${\tt NLnonzero}\in NP$. In \cite{DeLoera, Mulmuley} it is shown that the analogous
problem ${\tt LRnonzero}$ (deciding $c_{\lambda,\mu}^{\nu}>0$) can be done in
polynomial time. Their proof relies on the Saturation Theorem for $c_{\lambda,\mu}^{\nu}$.

Conjecture~\ref{theconj} implies 
${\tt NLnonzero}\in P$ as well. In brief, Conjecture~\ref{theconj} actually shows
\[N_{\mu,\nu,\lambda}\neq 0 \iff {\mathcal P}_{\mu,\nu,\lambda}\neq \emptyset.\] 
The ``$\Rightarrow$'' implication is by Theorem~\ref{thm:thepolytope}. For ``$\Leftarrow$'', we may assume, by Lemma~\ref{prop:easy}(V), that $|\lambda|+|\mu|+|\nu|\equiv 0 \ (\!\!\!\mod 2)$. Then ${\mathcal P}_{\mu,\nu,\lambda}\neq \emptyset$ implies
${\mathcal P}_{\mu,\nu,\lambda}$ contains a rational point $\vec p$. Then choose $k\in {\mathbb Z}_{>0}$ such that $k\cdot \vec p\in k{\mathcal P}_{\mu,\nu,\lambda}$ 
is a lattice point. By construction, 
$k{\mathcal P}_{\mu,\nu,\lambda}={\mathcal P}_{k\mu,k\nu,k\lambda}$
and so by Theorem~\ref{thm:thepolytope}, $N_{k\mu,k\nu,k\lambda}>0$. Conjecture~\ref{theconj}
then says $N_{\mu,\nu,\lambda}>0$. Finally, the inequalities defining the Newell-Littlewood polytope
are of the form $Ax\leq b$ where the entries of $A$ are $0,\pm 1$ whereas the entries of $b$ are integers.
Hence the polytope is \emph{combinatorial}, and one can appeal \'{E}.~Tardos' algorithm \cite{Grotschel, Tardos}
to decide if ${\mathcal P}_{\mu,\nu,\lambda}$ is feasible in strongly polynomial time. This completes the
conditional argument.

\section{Multiplicity-freeness}
\label{sec:multfree}
In Section~\ref{sec:polytope} we studied when
$N_{\lambda,\mu,\nu}=0$. We now look at a related problem, proving an analogue of
J.~R.~Stembridge's \cite[Theorem~3.1]{jrs} which characterizes pairs $(\mu,\nu)\in {\sf Par}$
such that (\ref{eqn:prodabc}) is multiplicity-free, \emph{i.e.}, 
$c_{\mu,\nu}^{\lambda}\in \{0,1\}$ for all $\lambda\in {\sf Par}$.

Call a pair $(\mu,\nu)\in {\sf Par}^2$ \emph{NL-multiplicity-free} if (\ref{eqn:polyver}) contains no multiplicity, 
\emph{i.e.}, each $N_{\mu,\nu,\lambda}\in\{0,1\}$ for all $\lambda\in {\sf Par}$. 

\begin{theorem}
\label{thm:mfrep}
A pair $(\mu,\nu)\in {\sf Par}^2$ is NL-multiplicity-free if and only if
\begin{itemize}
\item[(I)] $\mu$ or $\nu$ is either a single box or $\emptyset$;
\item[(II)] $\mu$ is a single row and $\nu$ is a rectangle (or vice versa); or
\item[(III)] $\mu$ is a single column and $\nu$ is a rectangle (or vice versa).
\end{itemize}
\end{theorem}

Before the proof, we pause to compare and contrast Theorem~\ref{thm:mfrep} with \cite[Theorems~3.1, 4.1]{jrs},
and with J.~R.~Stembridge's later work \cite{jrs:general}. Theorem~\ref{thm:mfrep}
is an analogue of \cite[Theorem~3.1]{jrs} in the sense that the Schur functions
$\{s_{\lambda}\}$ are universal characters for ${\sf GL}$, whereas $\{s_{[\lambda]}\}$
are universal characters for ${\sf Sp}$ (we repeat that by \cite[Theorem~2.3.4]{Koike},
Theorem~\ref{thm:mfrep} holds without change for {\sf SO}). A generalization of
\cite[Theorem~3.1]{jrs} is \cite[Theorem~4.1]{jrs}, which characterizes when a product of
Schur \emph{polynomials} $s_{\mu}(x_1,\ldots,x_n)s_{\nu}(x_1,\ldots,x_n)$ is multiplicity-free. This is a generalization since (\ref{eqn:thespecialization}) preserves 
multiplicity-freeness. 

Since $s_{\mu}(x_1,\ldots,x_n)$ is the character of the (finite) ${\sf GL}(V)$-module ${\mathbb S}_{\lambda}(V)$, \cite{jrs:general} provides the appropriate generalization to all other Weyl characters (associated to an irreducible representation
of a complex semisimple Lie algebra). However, unlike the ${\sf GL}$ story, the modification rules are non-positive
(see the discussion and references of Section~\ref{sec:earlierwork}). Nevertheless, by invoking \cite[Corollary~2.5.3]{Koike}, it should be possible to derive Theorem~\ref{thm:mfrep} 
from \cite{jrs:general} by translating the root-system language to partitions (we have not actually done this). That said,
our proof is different and self-contained, starting from (\ref{eqn:Newell-Littlewood}).
It is relatively short, and has a component 
(Lemma~\ref{lem:jrsidea}) which might be of some independent interest.

\begin{proof}
($\Leftarrow$) Suppose we are in case (I). If $\mu=\emptyset$, then $c_{\alpha,\beta}^\mu>0$ if and only if $\alpha=\beta= \emptyset$, in which case $c_{\alpha,\beta}^\mu = 1$.  Hence, $c_{\alpha,\gamma}^\nu=\delta_{\gamma,\nu}$. Therefore $N_{\emptyset,\nu,\lambda}=\delta_{\nu,\lambda}$.
 As a result, 
${s}_{[\emptyset]}   
 {s}_{[\nu]} = {s}_{[\nu]}$
 is multiplicity-free. Thus we may suppose $\mu=(1)$. This case is NL-multiplicity-free
 by Corollary~\ref{prop:boxcase}.
 
(III) follows from (II) by Lemma~\ref{prop:easy}(VI). 

Thus suppose we are in case (II). Without loss, let $\mu = (k)$ and let $\nu = (c^d)$.  
We apply Proposition~\ref{thm:Pieri}, and specifically (\ref{eqn:inotherwords}). Since $\nu$ is a rectangle, for any $0\leq j\leq k$ there is at most one way to remove a horizontal
strip of size $j$ from $\nu$. The result is a shape $\theta_u=(c^{d-1},u)$ where $0\leq u\leq c$. Straightforwardly, if $u\neq u'$ then one cannot add a horizontal strip of $k-j$ boxes to $\theta_u$
and separately to $\theta_{u'}$ and obtain the same $\lambda$. NL-multiplicity-freeness follows from this analysis.

($\Rightarrow$) Our argument is similar to (and uses) the one used in J.~Stembridge's
work \cite{jrs}. If $\alpha,\beta\in {\sf Par}$, by $\alpha\cup\beta$
we mean the partition obtained by sorting the (nonzero) parts in the multiset union
of $\alpha$ and $\beta$.

\begin{lemma}
\label{lem:jrsidea}
For all triples of partitions $\mu,\nu,\lambda$ and $t\in {\mathbb Z}_{\geq 0}$, \[N_{\mu\cup (t),\nu,\lambda\cup (t)}\geq N_{\mu,\nu,\lambda} \text{ and } N_{\mu+ (1^t),\nu,\lambda+ (1^t)}\geq N_{\mu,\nu,\lambda}.\]
\end{lemma}
\noindent
\emph{Proof of Lemma~\ref{lem:jrsidea}:}
We will only prove the first assertion; the second follows by Lemma~\ref{prop:easy}(VI).
By \cite[Lemma~2.2]{jrs},
\begin{equation}
\label{eqn:jrsidea}
c_{\sigma\cup (t),\pi}^{\kappa\cup (t)}\geq c_{\sigma,\pi}^{\kappa}.
\end{equation}
Compare
\begin{equation}
N_{\mu,\nu,\lambda}=\sum_{\alpha^{\bullet},\beta^{\bullet},\gamma^{\bullet}} c_{\alpha^{\bullet},\beta^{\bullet}}^{\mu}
c_{\alpha^{\bullet},\gamma^{\bullet}}^{\nu}c_{\beta^{\bullet},\gamma^{\bullet}}^{\lambda}
\end{equation}
with
\begin{equation}
N_{\mu\cup (t),\nu,\lambda\cup (t)}=\sum_{\alpha^{\circ},\beta^{\circ},\gamma^{\circ}} c_{\alpha^{\circ},\beta^{\circ}}^{\mu\cup (t)}
c_{\alpha^{\circ},\gamma^{\circ}}^{\nu}c_{\beta^{\circ},\gamma^{\circ}}^{\lambda\cup (t)}.
\end{equation}
Notice that if $(\alpha^{\bullet},\beta^{\bullet},\gamma^{\bullet})$ is a witness for 
$N_{\mu,\nu,\lambda}$ then by (\ref{eqn:jrsidea}), 
$(\alpha^{\circ},\beta^{\circ},\gamma^{\circ}):=(\alpha^{\bullet}, \beta^{\bullet}\cup (t), \gamma^{\bullet})$
is a witness for $N_{\mu\cup (t),\nu, \lambda\cup (t)}$, and moreover
$N_{\mu\cup(t),\nu,\lambda\cup (t)}\geq N_{\mu,\nu,\lambda}$, as desired.
\qed

Suppose $(\mu,\nu)\in {\sf Par}^2$ that do not fall into (I), (II), or (III).  We break the argument into two cases, depending on whether either of $\mu$ or $\nu$ is a rectangle.

\noindent \emph{Case 1: (One of $\mu$ or $\nu$ is not a rectangle)}  Say that $\nu$ is not a rectangle.  Since $\mu$ is not a single box, it has at least two rows or at least two columns.  In view of Lemma~\ref{prop:easy}(VI), we may assume without loss of generality that $\mu$ has at least two columns. We first establish:

\begin{claim}
\label{claim:onerow}
For $\nu$ not a rectangle and $k\geq 2$, $N_{(k),\nu,\nu+(k-2)} \geq 2$.
\end{claim}
\noindent\emph{Proof of Claim~\ref{claim:onerow}:}
Since $\nu$ is not a rectangle, it has two corners, so let $\alpha=(1)$, $\beta= (k-1)$, and $\gamma$ and 
${\overline{\gamma}}$ each be $\nu$ with a different corner removed.  By (\ref{eqn:usualPieri}),
\[c^{(k)}_{(1),(k-1)}=c^\nu_{\gamma,(1)}=c^\nu_{{\overline{\gamma}},(1)}=1,\] 
and since $(\nu+(k-2))/\gamma$ and $(\nu+(k-2))/{\overline{\gamma}}$ are horizontal strips of $k-1$ boxes, 
\[c^{\nu+(k-2)}_{\gamma,(k-1)}=c^{\nu+(k-2)}_{{\overline{\gamma}},(k-1)}=1.\] 
Therefore, 
\[N_{(k),\nu,\nu+(k-2)} \geq c^{(k)}_{(1),(k-1)}c^\nu_{\gamma,(1)}c^{\nu+(k-2)}_{\gamma,(k-1)} +  c^{(k)}_{(1),(k-1)}c^\nu_{{\overline{\gamma}},(1)}c^{\nu+(k-2)}_{{\overline{\gamma}},(k-1)} = 2,\]
as asserted.
\qed

In general, consider $\mu$ and $\nu$ such that $\mu_1\geq 2$, and $\nu$ is not a rectangle.  Let
$\lambda = (\nu+(\mu_1-2))\cup (\mu_2,\mu_3,\dots)$. 
By repeated application of Lemma \ref{lem:jrsidea}, followed by Claim \ref{claim:onerow}:
\[N_{\mu,\nu,\lambda} = N_{\mu,\nu, (\nu+(\mu_1-2))\cup (\mu_2,\mu_3,\dots)} \geq N_{(\mu_1,\mu_3,\mu_4,\dots),\nu,(\nu+(\mu_1-2))\cup(\mu_3,\dots)}\geq \dots \geq  N_{(\mu_1),\nu,\nu+(\mu_1-2)}\geq 2.\]
Hence $(\mu,\nu)$ is not NL-multiplicity-free.

\noindent \emph{Case 2: ($\mu$ and $\nu$ are both rectangles with at least two rows and columns)}  We first consider the special case $\mu=(k^2)$ and $\nu=(c^d)$:
\begin{claim}
\label{claim:tworows}
For $k,c,d\geq 2$, $N_{(k^2),(c^d),((c+k-2)^2)\cup (c^{d-2})}\geq 2$.
\end{claim}
\noindent\emph{Proof of Claim~\ref{claim:tworows}:}
Let $\alpha=(1,1), \beta=(k-1,k-1), \gamma = (c^{d-2})\cup ((c-1)^2)$. 
By the Littlewood-Richardson rule, 
\[c_{\alpha, \beta}^{(k^2)}=c_{\alpha, \gamma}^\nu= c_{\beta, \gamma}^{((c+k-2)^2)\cup (c^{d-2})}= 1.\]  
Similarly, letting ${\overline{\alpha}}=(2)$, ${\overline{\beta}}=(k,k-2)$, ${\overline{\gamma}} = (c^{d-1})\cup (c-2)$, we obtain
\[c_{{\overline{\alpha}}, {\overline{\beta}}}^{(k^2)}=c_{{\overline{\alpha}}, {\overline{\gamma}}}^\nu= c_{{\overline{\beta}}, {\overline{\gamma}}}^{((c+k-2)^2)\cup (c^{d-2})}= 1.\]  
Therefore,  \[N_{(k^2),(c^d),((c+k-2)^2)\cup (c^{d-2})} \geq c_{\alpha, \beta}^{(k^2)}c_{\alpha, \gamma}^\nu c_{\beta, \gamma}^{((c+k-2)^2)\cup (c^{d-2})} + c_{{\overline{\alpha}}, {\overline{\beta}}}^{(k^2)}c_{{\overline{\alpha}}, 
\overline{\gamma}}^\nu c_{\overline{\beta}, \overline{\gamma}}^{((c+k-2)^2)\cup (c^{d-2})} = 2,\] 
as needed.
\qed

Consider arbitrary rectangles $\mu=(k^p)$ and $\nu=(c^d)$ that both contain at least two rows and columns; hence $k,p,c,d\geq 2$.  Let
$\lambda = ((c+k-2)^2)\cup (k^{p-2})\cup (c^{d-2})$. 
By repeatedly applying Lemma \ref{lem:jrsidea}, followed by Claim \ref{claim:tworows}:
\begin{multline}\nonumber
N_{\mu,\nu,\lambda} = N_{(k^p),(c^d),((c+k-2)^2)\cup (k^{p-2})\cup (c^{d-2})} \geq N_{(k^{p-1}),(c^d),((c+k-2)^2)\cup (k^{p-3})\cup (c^{d-2})}\\\geq \dots \geq N_{(k^2),(c^d),((c+k-2)^2)\cup (c^{d-2})}\geq 2.
\end{multline}
Hence $(\mu,\nu)$ is not NL-multiplicity-free in this case, either.

These two cases cover all possibilities for $\mu$ and $\nu$ not satisfying (I), (II), or (III).
In both cases we established multiplicity.  
\end{proof}

\section{Final remarks}\label{sec:final}
\subsection{The associativity relation}
Since $N_{\mu,\nu,\lambda}$ are the structure constants for the Koike-Terada basis
of $\Lambda$, the associativity relation
\[({s}_{[\mu]}   {s}_{[\nu]} )  {s}_{[\lambda]} ={s}_{[\mu]}   ({s}_{[\nu]}   {s}_{[\lambda]} ),\] 
implies for any $\mu,\nu,\lambda,\tau\in {\sf Par}$ that:
\begin{equation}
\label{eqn:assoc}
\sum_{\theta} N_{\mu,\nu,\theta}N_{\theta,\lambda,\tau}=\sum_{\theta} N_{\nu,\lambda,\theta}N_{\mu,\theta,\tau}.
\end{equation}

\begin{problem}
Give a bijective proof of (\ref{eqn:assoc}) using the definition (\ref{eqn:Newell-Littlewood}).
\end{problem}

Now, $c_{\mu,\nu}^{\lambda}$ also ``associative'' in that it satisfies a relation of the form (\ref{eqn:assoc}).
However, (\ref{eqn:assoc}) does not \emph{formally} follow  from this fact. To explain,
we considered other associative 
structure coefficients $w_{\mu,\nu}^{\lambda}$ studied in algebraic combinatorics. For each of these one can define a ``Newell-Littlewood'' analogue:
\[O_{\mu,\nu,\lambda}:=\sum_{\alpha,\beta,\gamma} w_{\alpha,\beta}^{\mu}
w_{\alpha,\gamma}^{\nu}w_{\beta,\gamma}^{\lambda}.\]
Specifically, we looked at the
$K$-theoretic Littlewood-Richardson coefficients for Grassmannians, 
the shifted Littlewood-Richardson coefficients for multiplication of Schur $P-$ or Schur $Q-$ functions, and the structure coefficients for Schubert polynomials (here we replace
partitions with permutations). Small examples show $O_{\mu,\nu,\lambda}$ is not associative. 
Under what conditions/natural examples is $O_{\mu,\nu,\lambda}$ associative?

\subsection{An analogue of M.~Kleber's conjecture} Fix a rectangle $a\times b$ and consider all products 
$s_{\lambda}s_{\lambda^{\vee}}$ where $\lambda\subseteq a\times b$ and $\lambda^{\vee}$ is the 180-degree rotation of $(a\times b)\setminus \lambda$. M.~Kleber \cite[Section~3]{Kleber} conjectured that these products, ranging over unordered pairs
$(\lambda,\lambda^{\vee})$ are linearly independent in $\Lambda$. 

\begin{problem}\label{prob:kleber}
Are the products ${s}_{[\lambda]}  {s}_{[\lambda^{\vee}]}$, indexed over unordered pairs of partitions 
$(\lambda,\lambda^\vee)$ contained in $a\times b$, linearly independent in $\Lambda$? 
\end{problem}

By Lemma~\ref{prop:easy}(II), M.~Kleber's conjecture implies an affirmative 
answer to Problem~\ref{prob:kleber}.
However, the extra terms in ${s}_{[\lambda]}  {s}_{[\lambda^{\vee}]}$ versus $s_{\lambda}s_{\lambda^{\vee}}$ might make Problem~\ref{prob:kleber} more tractable. (The interested reader can test ideas for 
$a=b=2$ using the data in the Appendix.)

\subsection{Version of T.~Lam-A.~Postnikov-P.~Pylyavskyy's theorems}
We give another implication of Proposition~\ref{prop:reformulate}.
This concerns results of T.~Lam-A.~Postnikov-P.~Pylyavskyy \cite{AJM}. Their paper solves (and generalizes) conjectures of A.~Okounkov \cite{Okounkov} and S.~Fomin-W.~Fulton-C.-K.~Li-T.-Y.~Poon \cite{FominPoon}.
It builds on work of B.~Rhoades-M.~Skandera \cite{RS1, RS2}.

If $\alpha,\beta\in {\sf Par}$ then $\alpha\vee\beta\in {\sf Par}$ has parts
$\max(\alpha_i,\beta_i)$ (where we have adjoined $0$'s to $\alpha$ or $\beta$
as necessary). For any two skew shapes $\nu/\alpha$ and $\mu/\beta$, define 
\[(\nu/\alpha)\wedge(\mu/\beta):=(\nu \wedge \mu) /(\alpha \wedge \beta) \text{ \ 
and $(\nu/\alpha)\vee (\mu/\beta) := (\nu \vee \mu) / (\alpha \vee \beta)$.}\] 
Let
\[{\rm sort}_{1}(\nu,\mu) := (\rho_{1},\rho_{3},\rho_{5},\ldots)\text{\ and ${\rm sort}_{2}(\nu,\mu) := (\rho_{2},\rho_{4},\rho_{6},\ldots)$},\]
where $(\rho_{1},\rho_{2},\rho_{3},\ldots):=\nu\cup \mu$. Below, $\frac{\nu+\mu}{2}$ means coordinate-wise
addition and division. Also $\lfloor \cdot \rfloor$ and $\lceil \cdot \rceil$ are taken coordinate-wise.

If $f\in \Lambda$ then $f$ is said to be \emph{Schur nonnegative} if $f=\sum_{\lambda} a_{\lambda} s_{\lambda}$
with $a_{\lambda}\geq 0$ for all $\lambda \in {\sf Par}$.

\begin{theorem}[\cite{AJM}]\label{thm:sschurmj}
Let $\nu/\alpha$ and $\mu/\beta$ be skew shapes. The following are Schur nonnegative:
\begin{enumerate}
    \item $s_{(\nu/\alpha)\wedge(\mu/\beta)}s_{(\nu/\alpha)\vee (\mu/\beta)} - s_{\nu/\alpha}s_{\mu/\beta}$
    \item $s_{\lfloor \frac{\nu + \mu}{2} \rfloor / \lfloor \frac{\alpha + \beta}{2}\rfloor}s_{\lceil \frac{\nu + \mu}{2} \rceil / \lceil \frac{\alpha + \beta}{2}\rceil} - s_{\nu/\alpha}s_{\mu/\beta}$
    \item $s_{\textup{sort}_{1}(\nu,\mu)/\textup{sort}_{1}(\alpha,\beta)}s_{\textup{sort}_{2}(\nu,\mu)/\textup{sort}_{2}(\alpha,\beta)} - s_{\nu/\alpha}s_{\mu/\beta}$
\end{enumerate}
\end{theorem}

Define $f\in \Lambda$ to be \emph{Koike-Terada nonnegative} if $f=\sum_{\lambda} b_{\lambda}s_{[\lambda]}$
has $b_{\lambda}\geq 0$ for every $\lambda \in {\sf Par}$.

\begin{theorem}
\label{thm:meetjoin}
The following are  Koike-Terada nonnegative:
\begin{enumerate}
    \item ${s}_{[\nu\wedge \mu]}  
    {s}_{[\nu \vee\mu]} -{s}_{[\nu]}   {s}_{[\mu]} $
    \item ${s}_{[\lfloor \frac{\nu + \mu}{2} \rfloor]}    {s}_{[\lceil \frac{\nu + \mu}{2} \rceil]} -{s}_{[\nu]}  {s}_{[\mu]} $
    \item ${s}_{[\textup{sort}_{1}(\nu,\mu)]}    {s}_{[\textup{sort}_{2}(\nu,\mu)]}  - {s}_{[\nu]}   {s}_{[\mu]} $
\end{enumerate}
\end{theorem}
\begin{proof}
We only prove the first statement; the others are similar. Fix any $\lambda$. Then
\begin{align*}
N_{\mu,\nu,\lambda}& = [s_{\lambda}] \sum_{\alpha} s_{\mu/\alpha}s_{\nu/\alpha} & \text{(Proposition~\ref{prop:reformulate})}\\
\ & \leq [s_{\lambda}] \sum_{\alpha} s_{\mu\wedge \nu/\alpha}s_{\mu\vee \nu/\alpha} &
\text{(Theorem~\ref{thm:sschurmj}(1))}\\
& = N_{\mu\wedge \nu,\mu\vee\nu,\lambda} & \text{(Proposition~\ref{prop:reformulate})}
\end{align*}
and the result follows.
\end{proof}

\begin{example}
Let $\mu=(2), \nu=(1,1)$. Then
\[{s}_{[\mu]}   {s}_{[\nu]} =
{s}_{[2]}   {s}_{[1,1]} =
{s}_{[1, 1]}  + {s}_{[2]}  + {s}_{[2, 1, 1]}  + 
{s}_{[3, 1]} ,\]
and
\[
{s}_{[\mu\wedge \nu]}   {s}_{[\mu\vee\nu]} =
{s}_{[1]}   {s}_{[2,1]} =
{s}_{[1,1]}  + {s}_{[2]} 
 + {s}_{[2, 1, 1]}  + 
{s}_{[3, 1]} +{s}_{[2,2]}.\]
Hence 
\[{s}_{[\mu\wedge \nu]}   {s}_{[\mu\vee\nu]} -
{s}_{[\mu]}   {s}_{[\nu]} ={s}_{[2,2]},\] 
which
is ${s} $-positive, as asserted by Theorem~\ref{thm:meetjoin}(1). The reader can verify that, in this
case, 
\[s_{[\mu\wedge\nu]}s_{[\mu\vee \nu]}={s}_{[\lfloor \frac{\nu + \mu}{2} \rfloor]}    {s}_{[\lceil \frac{\nu + \mu}{2} \rceil]}={s}_{[\textup{sort}_{1}(\nu,\mu)]}    {s}_{[\textup{sort}_{2}(\nu,\mu)]}.\] 
Therefore the above also agrees with
parts (2) and (3) of Theorem~\ref{thm:meetjoin}, as well. \qed
\end{example}

\section*{Acknowledgements}
We thank Frank Calegari, Joshua Kiers and William Linz for helpful conversations. Steven Sam's 2010 blog post
on the numbers offered a useful resource when starting this project.
We thank Richard Rimanyi for crucial computer assistance. We used Anders Buch's Littlewood-Richardson calculator in our experiments. AY was partially supported by a Simons Collaboration grant and funding from UIUC's Campus Research Board.

\appendix
\section{A list of products $s_{[\mu]}s_{[\nu]}$}
We compute (\ref{eqn:polyver})
for $\emptyset\neq\mu,\nu\subseteq 2\times 2$.
\[{s}_{[1]}^{2}={s}_{[0]}  + {s}_{[1, 1]}  + {s}_{[2]} \]
\[{s}_{[1]}   {s}_{[2]} ={s}_{[1]}  + {s}_{[2, 1]}  + {s}_{[3]} \]
\[{s}_{[1]}   {s}_{[1,1]} ={s}_{[1]}  + {s}_{[1,1,1]}  + {s}_{[2,1]} \]
\[{s}_{[1]}   {s}_{[2,1]} ={s}_{[1,1]}  + {s}_{[2]}  + {s}_{[2,1,1]} + {s}_{[2,2]} + {s}_{[3,1]} \]
\[{s}_{[1]}   {s}_{[2,2]} ={s}_{[2,1]}  + {s}_{[2,2,1]}  + {s}_{[3,2]} \]
\[{s}_{[2]} ^{  2}={s}_{[0]}  + {s}_{[1,1]}  + {s}_{[2]} + {s}_{[2,2]} + {s}_{[3,1]} + {s}_{[4]} \]
\[{s}_{[2]}   {s}_{[1,1]} ={s}_{[1,1]}  + {s}_{[2]}  + {s}_{[2,1,1]} + {s}_{[3,1]} \]
\[{s}_{[2]}   {s}_{[2,1]} ={s}_{[1]}  + {s}_{[1,1,1]}  + 2{s}_{[2,1]}+ {s}_{[3]} + {s}_{[2,2,1]} + {s}_{[3,1,1]} + {s}_{[3,2]} + {s}_{[4,1]}\]
\[{s}_{[2]}   {s}_{[2,2]} ={s}_{[2]}  + {s}_{[2,1,1]}  + {s}_{[2,2]} + {s}_{[3,1]} + {s}_{[2,2,2]} + {s}_{[3,2,1]} 
+ {s}_{[4,2]}\]
\[{s}_{[1,1]} ^{  2}={s}_{[0]}  + {s}_{[1,1]}  + {s}_{[2]} + {s}_{[1,1,1,1]} 
 + {s}_{[2,1,1]} + {s}_{[2,2]} \]
\[
{s}_{[1,1]}   {s}_{[2,1]} ={s}_{[1]}  + 2{s}_{[2,1]}  + {s}_{[3]} + {s}_{[2,1,1,1]} + {s}_{[2,2,1]} + {s}_{[3,1,1]} + {s}_{[3,2]} \]
\[{s}_{[1,1]}   {s}_{[2,2]} ={s}_{[1,1]}  + {s}_{[2,1,1]}  + {s}_{[2,2]} + {s}_{[3,1]} + {s}_{[2,2,1,1]} + {s}_{[3,2,1]} + {s}_{[3,3]} \]
\begin{multline}\nonumber
{s}_{[2,1]} ^{  2}={s}_{[0]}  + 2{s}_{[1,1]} + 2{s}_{[2]}+ {s}_{[1,1,1,1]} + 3{s}_{[2,1,1]}
+ 2{s}_{[2,2]}+ 3{s}_{[3,1]} 
+ {s}_{[4]}  + {s}_{[2,2,1,1]} + {s}_{[2,2,2]} + {s}_{[3,1,1,1]} \\ \nonumber
+ 2{s}_{[3,2,1]} + {s}_{[3,3]} 
+ {s}_{[4,1,1]} + {s}_{[4,2]} 
\end{multline}
\begin{multline}\nonumber
{s}_{[2,1]}   {s}_{[2,2]} ={s}_{[1]}  + {s}_{[1,1,1]}  + 2{s}_{[2,1]} + {s}_{[3]} + {s}_{[2,1,1,1]} 
+ 2{s}_{[2,2,1]} + 2{s}_{[3,1,1]} 
+ 2{s}_{[3,2]}  + {s}_{[4,1]} + 
{s}_{[2,2,2,1]} + {s}_{[3,2,1,1]} \\ \nonumber
+ {s}_{[3,2,2]} + {s}_{[3,3,1]} 
+ {s}_{[4,2,1]} + {s}_{[4,3]} 
\end{multline}
\begin{multline}\nonumber
{s}_{[2,2]} ^{  2}={s}_{[0]}  + {s}_{[1,1]}  + {s}_{[2]} + {s}_{[1,1,1,1]} 
+ {s}_{[2,1,1]} 
+ 2{s}_{[2,2]} 
+ {s}_{[3,1]} 
+ {s}_{[4]}  + {s}_{[2,2,1,1]} + {s}_{[2,2,2]} + {s}_{[3,1,1,1]} 
+ 2{s}_{[3,2,1]}\\ \nonumber
+ {s}_{[3,3]} 
+ {s}_{[4,1,1]} + {s}_{[4,2]} 
+ {s}_{[2,2,2,2]} + {s}_{[3,2,2,1]} 
+ {s}_{[3,3,1,1]} + {s}_{[4,2,2]} 
+ {s}_{[4,3,1]} + {s}_{[4,4]} 
\end{multline}

The computation ${s}_{[2]} {s}_{[2,2]} $
matches the multiplication $(2,2)_{\sf Sp}\times (2)_{\sf Sp}$ in 
\cite[pg. 509]{Koike}. This calculation is coincides with the tensor products in ${\sf Sp}_{2n}$ for any $n\geq 3$. However, when $n=2$,
as shown in \emph{loc.~cit.} the expansion differs from the one above (and from each other, among the classical groups).
\end{document}